\documentclass[11pt,reqno]{amsart}

\usepackage[utf8]{inputenc}
\usepackage[english]{babel}
\usepackage{amsaddr}
\usepackage{comment}
\usepackage{thm-restate}

\usepackage[margin=1in]{geometry}
\usepackage{amsmath, amsthm, amsfonts, amssymb, tikz-cd, mathtools, mathrsfs,theoremref, txfonts}
\usepackage[T1]{fontenc}
\usepackage{graphicx}
\usepackage{hyperref}

\theoremstyle{definition}
\newtheorem{theorem}{Theorem}[section]

\newtheorem{corollary}[theorem]{Corollary}
\newtheorem{lemma}[theorem]{Lemma}

\newtheorem{claim}[theorem]{Claim}

\theoremstyle{definition}
\newtheorem{definition}[theorem]{Definition}

\newtheorem{question}[theorem]{Question}

\DeclareMathOperator{\Popt}{Pop_T}
\DeclareMathOperator{\Pops}{Pop_S}
\DeclareMathOperator{\aexc}{aexc}
\DeclareMathOperator{\Aexc}{Aexc}
\DeclareMathOperator{\cyc}{cyc_{>1}}
\DeclareMathOperator{\sgn}{sgn}

\usepackage{tikz}
\usetikzlibrary{fit, calc, positioning, cd, chains, shapes}
\newcommand{\equic}[2][1 cm]{
  \draw (0,0) circle (#1);
  \foreach \i in {1,...,#2} {
    \coordinate (N\i) at (-\i*360/#2:#1);
    \fill[black] (N\i) circle (0.05 cm) node[inner sep=4pt,text depth=2pt, shift = (-\i*360/#2:0.3)]{\i};
  }
  }
 \newcommand{\equib}[2][1 cm]{
  \draw (0,0) circle (#1);
  \foreach \i [evaluate=\i as \j using {-\i}] in {1,...,#2} {
    \coordinate (N\i) at (-\i*180/#2:#1);
    \fill[black] (N\i) circle (0.05 cm) node[shift = (-\i*180/#2:0.3)]{\i};
    \coordinate (P\i) at (\j*180/#2+180:#1);
    \fill[black] (P\i) circle (0.05 cm) node[shift = (\j*180/#2+180:0.3)]{\pgfmathprintnumber{\j}};
  }
  }
\newcommand{\equicnl}[2][1 cm]{
  \draw (0,0) circle (#1);
  \foreach \i in {1,...,#2} {
    \coordinate (N\i) at (-\i*360/#2:#1);
    \fill[black] (N\i) circle (0.05 cm);
  }
  }
  
\newcommand{\equid}[2][1 cm]{
  \draw (0,0) circle (#1);
  \foreach \i [evaluate=\i as \j using {-\i}] in {1,...,#2} {
    \coordinate (N\i) at (-\i*180/#2:#1);
    \fill[black] (N\i) circle (0.05 cm) node[shift = (-\i*180/#2:0.3)]{\i};
    \coordinate (P\i) at (\j*180/#2+180:#1);
    \fill[black] (P\i) circle (0.05 cm) node[shift = (\j*180/#2+180:0.3)]{\pgfmathprintnumber{\j}};
  }
  \coordinate (ori) at (0,0);
  \pgfmathtruncatemacro{\cent}{#2+1};
  \fill[black] (ori) circle (0.05cm) node[right]{\pgfmathprintnumber{\cent}};
  \draw (ori) node[left]{\pgfmathprintnumber{-\cent}};
  }
  
  \newcommand{\equidd}[2][1 cm]{
  \draw (0,0) circle (#1);
  \foreach \i [evaluate=\i as \j using {-\i}] in {1,...,#2} {
    \coordinate (N\i) at (-\i*180/#2:#1);
    \fill[black] (N\i) circle (0.05 cm) node[shift = (-\i*180/#2:0.3)]{\i};
    \coordinate (P\i) at (\j*180/#2+180:#1);
    \fill[black] (P\i) circle (0.05 cm) node[shift = (\j*180/#2+180:0.3)]{\pgfmathprintnumber{\j}};
  }
  \coordinate (ori) at (0,0);
  \pgfmathtruncatemacro{\cent}{#2+1};
  \fill[black] (ori) circle (0.05cm) node[below]{\pgfmathprintnumber{\cent}};
  \draw (ori) node[above]{\pgfmathprintnumber{-\cent}};
  }

\title{Dynamics of Pop-Tsack Torsing}
\author{Anqi Li}
\address{Massachusetts Institute of Technology, Cambridge, MA 02139, USA}
\email{anqili@mit.edu}

\begin{document}

\maketitle

\begin{abstract}
For a finite irreducible Coxeter group $(W,S)$ with a fixed Coxeter element $c$ and set of reflections $T$, Defant and Williams define a pop-tsack torsing operation $\Popt\colon W \to W$ given by $\Popt(w) = w \cdot \pi_T(w)^{-1}$ where $\pi_T(w) = \bigvee_{t \leq_{T}w, \ t \in T}^{NC(w,c)}t$ is the join of all reflections lying below $w$ in the absolute order in the non-crossing partition lattice $NC(w,c)$. This is a ``dual'' notion of the pop-stack sorting operator $\Pops$ introduced by Defant as a way to generalize the pop-stack sorting operator on $\mathfrak{S}_n$ to general Coxeter groups. Define the forward orbit of an element $w \in W$ to be $O_{\Popt}(w) = \{w, \Popt(w), \Popt^2(w), \ldots \}$. Defant and Williams established the length of the longest possible forward orbits $\max_{w \in W}|O_{\Popt}(w)|$ for Coxeter groups of coincidental types and type $D$ in terms of the corresponding Coxeter number of the group. In their paper, they also proposed multiple conjectures about enumerating elements with near maximal orbit length. We resolve all the conjectures that they have put forth about enumeration, and in the process we give complete classifications of these elements of Coxeter groups of types $A,B$ and $D$ with near maximal orbit lengths. 
\end{abstract}

\section{Introduction}

The study of \emph{combinatorial dynamics} investigates combinatorially defined operations on sets of objects. Specifically, given a set $X$ and a function $f \colon X \to X$, it is of interest to study the \emph{forward orbit} of an element $x \in X$ denoted as $O_f(x) = \{ x, f(x), f^2(x), \ldots \}$. A natural question that one might ask is what the distribution of $|O_f(x)|$ looks like. 

One well-studied class of combinatorial operations is that of sorting operations, where for the base set $X$ we take the symmetric group $\mathfrak{S}_n$ -- the set of all permutations of $[n] := \{1, 2, \ldots, n \}$. There are many well-known sorting procedures in the literature (see for instance \cite{ABBHL19,ABH21,B03,CF21,K73} and the many references therein). Of most relevance to us is the \emph{pop-stack sorting} operator $\Pops$, which reverses all of the descending runs of a permutation while keeping different descending runs in the same relative order. Ungar \cite{U82} proved that $\max_{\pi \in \mathfrak{S}_n}|O_{\Pops}(\pi)| = n$. The properties of the pop-stack sorting operator have recently been of interest to enumerative combinatorialists, yielding more modern proofs of the maximum length of forward orbits and other properties of $\Pops$ \cite{ABBHL19,ABH21,CG19,EG21}.

Defant \cite{D22} gives a natural generalization of the $\Pops$ operator to an arbitrary Coxeter group $W$. When $W = A_{n-1}$, this $\Pops$ operator recovers the usual pop-stack sorting operator on $\mathfrak{S}_n$. We refer the interested reader to \cite[Section 1, Section 2.2]{D22} for more details on the motivation for this construction.

We briefly describe this generalization here. Let $(W,S)$ be a finite Coxeter group with corresponding set of simple reflections $S$. Denote the left weak order on $W$ by $\leq_S$. Let $w_{\circ}$ be the longest element of $W$. For $J \subset S$, let $w_{\circ}(J)$ be the longest element in the parabolic subgroup of $W$ generated by $J$. We define the longest-element projection $\pi_S \colon W \to \{ w_{\circ}(J): J \subset S \} $ as \[\pi_S(w) = \bigvee_{s \leq_{S} w, s \in S}s,\] where the join is taken in the weak order lattice. With this notation in place, we define the generalized pop-stack sorting operation $\Pops: W \rightarrow W$ as $\Pops(w) := w \cdot \pi_S(w)^{-1}$. 

In \cite{D22}, Defant gives a generalization of Ungar's theorem. Defant proves $\sup_{w \in W}|O_{\Pops}(w)| = h$, where $h$ is the Coxeter number of $W$. Of note is that Defant's proof is type-independent, and avoids passing to combinatorial models of finite Coxeter groups. 

Expanding on the classical theory of Coxeter groups, recent work \cite{A09,Be03} in this area considers substituting the set of simple reflections $S$ with a larger generating set. More precisely, we can consider the conjugate closure $T:= \{w sw^{-1}: s \in S, w \in W \}$ of $S$, usually called the generating set of all reflections. In analogy with the weak order $\leq_S$, there exists an \emph{absolute order} $\leq_T$. Let $\ell_T(w)$ for $w \in W$ denote the minimum $T$-word length of $w$. The \emph{absolute order} $\leq_T$ on $W$ is defined by $\pi \leq_T \mu$ if and only if $\ell_T(\pi) = \ell_T(\mu) + \ell_T(\pi^{-1} \mu)$. 

It turns out in general that $\leq_T$ does not contain a maximal element. In order to construct a lattice out of $\leq_T$, we define the poset of \emph{non-crossing partitions} relative to a Coxeter element $c$ in $W$, by $NC(W,c):=[e,c] := \{ w \in W: e \leq_T w \leq_T c\}$. Brady and Watt \cite{BW08} gave the first uniform proof that $NC(W,c)$ is indeed a lattice. Bessis \cite{Be03} constructs a so-called ``dual'' Coxeter group where we consider $(W,T)$, and replace the left weak order with $NC(W,c)$. With this set-up, Defant and Williams \cite{DN21} define a ``dual'' version of $\Pops$ which they call \emph{pop-tsack torsing} and denote by $\Popt$. In what follows, fix a Coxeter element $c$ of $W$.  The \emph{non-crossing projection} $\pi_T\colon W \to NC(W,c)$ is defined to be 
\[ \pi_T(w,c) = \bigvee_{t \leq_T w, t \in T} t\]
where the join is taken in $NC(W,c)$. 

\begin{definition}
The \emph{Coxeter pop-tsack torsing} operator $\Popt\colon W \to W $ is the map
\[ \Popt(w,c) = w \cdot \pi_T(w,c)^{-1}.\]
\end{definition}

We may ask similar questions about the forward orbits under the $\Popt$ operator as we did in the context of $\Pops$. The astute reader may notice that our definition of $\Popt$ is a priori dependent on the choice of the Coxeter element $c$. Defant and Williams \cite{DN21} prove that the dynamical structure of $\Popt$ does not depend on $c$, leaving us free to pick $c$ in a way that simplifies our analysis.

The main results in \cite{DN21} are the following analogues of Defant's result on maximum forward orbit length for $\Pops$ in the context of $\Popt$.  

\begin{theorem}[{\cite[Theorem 5.1]{DN21}}]
Let $W$ be a Coxeter group of coincidental type with a fixed Coxeter element $c$ and Coxeter number $h$. Then:
\begin{itemize}
    \item $\Popt^{h-1}(w) = e$ for all $w \in W$.
    \item $\max_{w \in W}|O_{\Popt}(w)| = h$.
    \item The only forward orbit of size $h$ is $O_{\Popt}(c^{-1})$. 
\end{itemize}
\end{theorem}

\begin{theorem}[{\cite[Theorem 6.1]{DN21}}]
Let $W$ be a Coxeter group of type $D_n$. Then:
\begin{itemize}
    \item $\Popt^{2n-3}(w) = e$ for all $w \in W$.
    \item $\max_{w \in W}\left| O_{\Popt}(w)\right| = 2n-2$. 
\end{itemize}
\end{theorem}

Defant and Williams also establish results about $\Popt$ forward orbits for Coxeter groups of exceptional types. We refer the interested reader to \cite[Section 7]{DN21} for further discussion.

Unlike in \cite{D22}, where the proofs were type-independent, the proofs of \cite[Theorem 5.1, Theorem 6.1]{DN21} rely on the combinatorial model description of the non-crossing partition lattice in each of the settings. Another observation is that while \cite[Theorem 5.1]{DN21} establishes that there is a unique forward orbit of maximum length $h$, in type $D_n$ (see Section~\ref{sec:D}), where $h = 2n-2$, there could be multiple forward orbits of maximum length. For example, in $D_4$ there are 7 forward orbits of maximum length 6. It is therefore of interest to enumerate elements in $D_n$ with maximum possible forward orbit length. 

Defant and Williams \cite{DN21} make conjectures for the number of elements with forward orbits of size $h$ or $h- 1$ when the Coxeter group $W$ is of types $A$ and $B$. They also conjecture a formula for the number of forward orbits of size $h$ when $W$ is of type $D$. In this paper, we prove these conjectures.

\begin{definition}\label{def:aexc-A}
An \emph{antiexceedance} of a permutation $w \in \mathfrak{S}_n$ is an element $i \in [n]$ such that $i < w^{-1}(i)$. 
\end{definition}

\begin{definition}\label{def:a-split}
If $w \in \mathfrak{S}_n$ has exactly $n-2$ antiexceedances and $\cyc(\pi_T(w)) = 1$ then we say that $w$ is \emph{$A$-splittable}.
\end{definition}

\begin{restatable}{thm}{typeA}\label{thm:type-A}
An element $w \in \mathfrak{S}_n$ satisfies $\left| O_{\Popt}(w) \right| = n-1$ if and only if $w$ is $A$-splittable. In particular, there are $2^n - \binom{n}{2} - 1$ permutations with $\Popt$ forward orbit of length $n-1$. 
\end{restatable}

Recall that the \emph{hyperoctahedral group} $B_n$ is the group of permutations $w$ of $\pm[n]$ such that $w(\bar{i}) = \overline{w(i)}$, where for notational simplicity we denote $-i$ by $\bar{i}$.  Define a total order on $\pm[n]$ by 
\[ \bar{1} \prec \bar{2} \prec \cdots \prec \bar{n} \prec 1 \prec 2 \prec \cdots \prec n. \]
In this setting, we have the following notion of antiexceedance, paralleling the one in $A_n$, where we replace $<$ with $\prec$.

\begin{definition}\label{def:aexc-b}
An \emph{antiexceedance} of a permutation $w \in B_n$ is an element $i \in [n]$ such that $i \prec w^{-1}(i)$. 
\end{definition}

\begin{definition}\label{def:b-split}
If $w \in B_n$ has exactly $2n-2$ antiexceedances and $\cyc(\pi_T(w)) = 1$ then we say that $w$ is \emph{$B$-splittable}.

\end{definition}

\begin{restatable}{thm}{typeB}\label{thm:type-B}
An element $w \in B_n$ satisfies $\left| O_{\Popt}(w) \right| = 2n-1$ if and only if $w$ is $B$-splittable. In particular, there are $2^n - n - 1$ permutations with $\Popt$ forward orbit of length $2n-1$. 
\end{restatable}

Not only do we enumerate the permuations with the given forward orbit length, we also give complete characterizations of these permutations. The formula given in \cite[Conjecture 6.6]{DN21} is unfortunately incorrect; while Defant and Williams conjecture that the number of elements $w \in D_n$ with $|O_{\Popt}(w)| = 2n-2$ is $n(2^{n-1}-2)+1$, we show here that the actual value is $ (n-1)(2^{n-2} - 2)+1$. In the following theorem, we also characterize the permutations exhibiting this property, but defer further description to Section 4.

\begin{restatable}{thm}{typeD}\label{thm:type-D}
Fix the Coxeter element $c = (\overline{1} \enspace \overline{2} \cdots \overline{n-1} \enspace 1 \enspace 2 \enspace \cdots \enspace  n-1)(\overline{n} \enspace n)$. An element $w \in D_n$ satisfies $\left| O_{\Popt}(w) \right| = 2n-2$ if and only if either $w$ is a pre-$D$-splittable permutation or $w = c^{-1}$. In particular, the number of elements of $D_n$ that require exactly $2n-3$ iterations of $\Popt$ to reach the identity is $(n-1)(2^{n-2} - 2)+1$. 
\end{restatable}

\medskip
\noindent\textbf{Outline.} In Section 2 we study type A Coxeter groups and in Section 3 we handle type B Coxeter groups. The argument for type B Coxeter groups is near verbatim that for type A. As noted in \cite{DN21}, type $A_{2n-1}$ folds to $B_n$ (refer to \cite[Section 4.3]{DN21} and references therein for more information on folding). It is interesting to ask if there is a direct way to do enumeration via folding, though in our case we choose to do the counting directly. In Section 4 we handle the case of type D Coxeter groups, which is more complex than types A and B. We end off with some avenues for further research in Section 5, drawing inspiration from analogous questions in the setting of $\Pops$.

\section{Type A}

The Coxeter group of type $A_{n-1}$ is effectively the symmetric group $\mathfrak{S}_n$. We fix the Coxeter element $c = (1 \enspace 2 \cdots n)$ in this section. 

Now, we briefly describe the combinatorial model for the noncrossing partition lattice $NC(\mathfrak{S}_n, c)$. Label $n$ equidistant points on a circle by $1, 2, \ldots, n$ in clockwise order. 

\begin{definition}
A set partition $\pi$ of $[n]$ is a \emph{noncrossing set partition} if the convex hulls of the different blocks of $\pi$ do not have a common interior.
\end{definition}

Given a noncrossing set partition $\pi$, we can recover a noncrossing partition in $NC(\mathfrak{S}_n, c)$ as follows: for each block of $\pi$, form a cycle by ordering the elements of the block in a clockwise order around the circle. In fact, this procedure we have described gives a bijection; namely, every noncrossing partition in $NC(\mathfrak{S}_n, c)$ arises in such a fashion. The partial order on $NC(\mathfrak{S}_n,c)$ corresponds to the reverse refinement order on noncrossing set partitions. In particular, the noncrossing projection of a permutation $w \in \mathfrak{S}_n$ is the smallest noncrossing partition $v \in NC(\mathfrak{S}_n, c)$ such that every cycle in $w$ is contained (as a set) in a cycle of $v$. See Figure~\ref{fig:typeA-waexc} for an example. 

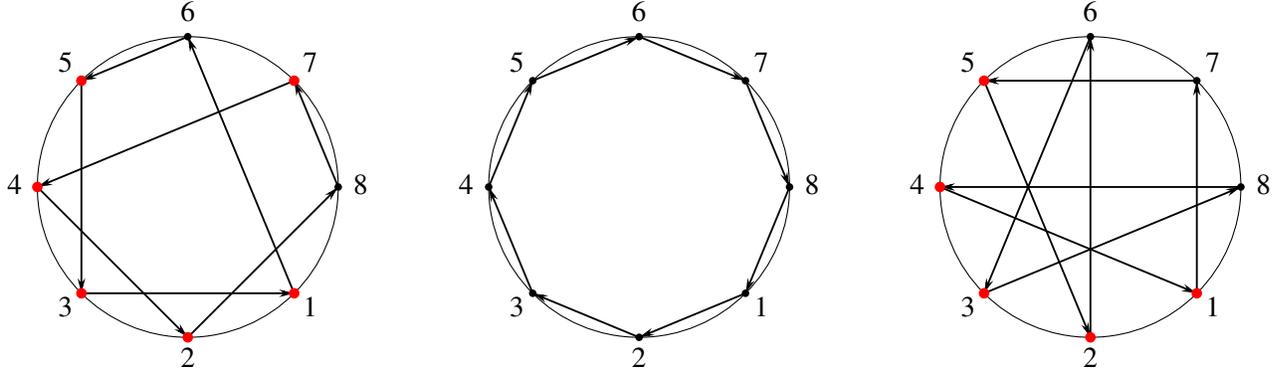
\begin{figure}[!htp]
    \centering
    \begin{tikzpicture}
\begin{scope}
\equic[2 cm]{8};
\draw [-{Stealth[length=2mm, width=1mm]}, line width = 0.7pt] (N1) -- (N6);
\draw [-{Stealth[length=2mm, width=1mm]}, line width = 0.7pt] (N6) -- (N5);
\draw [-{Stealth[length=2mm, width=1mm]}, line width = 0.7pt] (N5) -- (N3);
\draw [-{Stealth[length=2mm, width=1mm]}, line width = 0.7pt] (N3) -- (N1);

\draw [-{Stealth[length=2mm, width=1mm]}, line width = 0.7pt] (N8) -- (N7);
\draw [-{Stealth[length=2mm, width=1mm]}, line width = 0.7pt] (N7) -- (N4);
\draw [-{Stealth[length=2mm, width=1mm]}, line width = 0.7pt] (N4) -- (N2);
\draw [-{Stealth[length=2mm, width=1mm]}, line width = 0.7pt] (N2) -- (N8);

\foreach \i in {1,2,3,4,5,7}{
    \fill[red](N\i) circle (0.07 cm);
}

\end{scope}

\begin{scope}[shift={(6,0)}]

\equic[2 cm]{8};
\foreach \i in {1,...,7}{
    \pgfmathparse{\i+1}
    \edef\j{\pgfmathresult}
    \draw [-{Stealth[length=2mm, width=1mm]}, line width = 0.7pt] (N\i) -- (N\j);
}
\draw [-{Stealth[length=2mm, width=1mm]}, line width = 0.7pt] (N8) -- (N1);

\end{scope}

\begin{scope}[shift={(12,0)}]

\equic[2 cm]{8};
\draw [-{Stealth[length=2mm, width=1mm]}, line width = 0.7pt] (N1) -- (N7);
\draw [-{Stealth[length=2mm, width=1mm]}, line width = 0.7pt] (N7) -- (N5);
\draw [-{Stealth[length=2mm, width=1mm]}, line width = 0.7pt] (N5) -- (N2);
\draw [-{Stealth[length=2mm, width=1mm]}, line width = 0.7pt] (N2) -- (N6);
\draw [-{Stealth[length=2mm, width=1mm]}, line width = 0.7pt] (N6) -- (N3);
\draw [-{Stealth[length=2mm, width=1mm]}, line width = 0.7pt] (N3) -- (N8);
\draw [-{Stealth[length=2mm, width=1mm]}, line width = 0.7pt] (N8) -- (N4);
\draw [-{Stealth[length=2mm, width=1mm]}, line width = 0.7pt] (N4) -- (N1);

\foreach \i in {1,2,3,4,5}{
    \fill[red](N\i) circle (0.07 cm);
}
\end{scope}
\end{tikzpicture} 
    \caption{The permutation $w:= (8 \enspace 7 \enspace 4 \enspace 2)(6 \enspace 5 \enspace 3 \enspace 1) \in \mathfrak{S}_8$ (left), its non-crossing projection $\pi_T(w) = (1 \enspace 2 \enspace 3 \enspace 4 \enspace 5 \enspace 6 \enspace 7 \enspace 8) \in \mathfrak{S}_8$ (middle), and its image under $\Popt$ given by $\Popt(w) = (1 \enspace 7 \enspace 5 \enspace 2 \enspace 6 \enspace 3 \enspace 8 \enspace 4) \in \mathfrak{S}_8$ (right)}
    \label{fig:typeA-waexc}
\end{figure}

Next, we recall some results from \cite{DN21}. In the following, fix the Coxeter element $c = (1 \enspace 2 \enspace \cdots \enspace n)$ of $A_{n-1}$. We will work with $\Popt$ defined using this $c$; this is without loss of generality since the orbit structure $\Popt$ is independent of the choice of $c$ \cite[Corollary 4.5]{DN21}.

Following \cite{DN21}, we write $\aexc(w)$ for the number of antiexceedances (as per Definition~\ref{def:aexc-A}) in the permutation $w$ and we write $\Aexc(w)$ for the set of indices that are antiexceedances. Let $\cyc(w) \in \mathbb{N}$ denote the number of non-singleton cycles in a permutation $w \in \mathfrak{S}_n$. In this notation, we have the following result from \cite{DN21}. An illustration of this result is in Figure~\ref{fig:typeA-waexc}, where the antiexceedances are the vertices marked in red. 

\begin{lemma}[{\cite[Theorem 5.2]{DN21}}] \label{lem:DN-A}
The following properties relating antiexceedances and $\Popt$ hold:
\begin{enumerate}
    \item[(a)] The element $c^{-1}$ has $n-1$ antiexceedances. 
    \item[(b)] Every element of $\mathfrak{S}_n$ other than $c^{-1}$ has at most $n-2$ antiexceedances.
    \item[(c)] For every $w \in \mathfrak{S}_n$, we have $\aexc(\Popt(w)) = \aexc(w) - \cyc(\pi_T(w))$. 
\end{enumerate}
\end{lemma}

Since the only element of $\mathfrak{S}_n$ without any cycles of length greater than 1 is $e$, for any $w \in \mathfrak{S}_n \backslash \{ e \}$ we have that $\aexc(\Popt(w)) \leq \aexc(w) - 1 $. In order for $w$ to take $n-2$ iterations to reach $e$, by \cite[Theorem 5.1]{DN21} we must first have that $w \neq c^{-1}$. Then, by Lemma~\ref{lem:DN-A}(b) and the earlier observation in this paragraph, it follows that $w$ must have exactly $n-2$ antiexceedances and $\pi_T(w)$ needs to consist of exactly one cycle with length greater than 1. Recall that such elements $w\in A_n$ are $A$-splittable (Definition~\ref{def:a-split}). An example of an $A$-splittable permutation is the $w$ illustrated in Figure~\ref{fig:typeA-waexc}. We have seen that for an element $w \in \mathfrak{S}_n$ to have a $\Popt$ forward orbit of length $n-1$, $w$ has to be splittable. In fact, we will prove that this splittable condition is sufficient. We recall Theorem~\ref{thm:type-A}.

\typeA*

Combining Theorem~\ref{thm:type-A} with \cite[Theorem 5.1]{DN21}, we prove \cite[Conjecture 5.5]{DN21}. 

\begin{corollary}[{\cite[Conjecture 5.5]{DN21}}]\label{cor:typeA}
In type $A_{n-1}$, the number of permutations that require exactly $n-2$ or $n-1$ iterations
of $\Popt$ to reach the identity is $2^n - \binom{n}{2}$. 
\end{corollary}

The proof of Theorem~\ref{thm:type-A} is split into two parts. In the first, we will prove the claim that $A$-splittable permutations are precisely the elements with $\Popt$ forward orbit of length $n-1$. In the second, we will enumerate the $A$-splittable permutations.

One may be tempted to claim that if $\cyc(\pi_T(w)) = 1$, then $\cyc(\pi_T(\Popt(w)))=1$. Such a claim, if true, would immediately give Theorem~\ref{thm:type-A} upon iteration. Alas, this is too na\"ive a hope. For instance, one can check that $w = (1\enspace 4 \enspace 5 \enspace 8 \enspace 7 \enspace 6 \enspace 9 \enspace 3 \enspace 2)$ is a counterexample to the aforementioned claim; we have that $\pi_T(w) = (1 \enspace 2 \enspace 3 \enspace 4 \enspace 5 \enspace 6 \enspace 7 \enspace 8 \enspace 9)$ and so $\cyc(\pi_T(w)) = 1$, while $\Popt(w) = (1 \enspace 3) (2 \enspace 4)(6 \enspace 8)(7 \enspace 9)(5)$ and $\pi_T(\Popt(w)) = (1 \enspace 2 \enspace 3 \enspace 4)(6 \enspace 7 \enspace 8 \enspace 9)(5)$ and so $\cyc(\pi_T(\Popt(w))) = 2$. A pictorial illustration of this phenomenon is in Figure~\ref{fig:A-cex}. This suggests that we need to account for the structure inherent in $w$ from it being $A$-splittable and therefore having $n-2$ antiexceedances. For instance, in the earlier example it can be checked that $w$ is not $A$-splittable and also does not lie in the orbit of a $A$-splittable permutation. The latter can for instance be checked by computing that there does not exist $u$ with $\cyc(\pi_T(u)) = 1$ such that $\Popt(u) = w$. 

\begin{figure}
    \centering
    \begin{tikzpicture}
\begin{scope}
\equic[1.5cm]{9};
\draw [-{Stealth[length=2mm, width=1mm]}, line width = 0.7pt] (N1) -- (N4);
\draw [-{Stealth[length=2mm, width=1mm]}, line width = 0.7pt] (N4) -- (N5);
\draw [-{Stealth[length=2mm, width=1mm]}, line width = 0.7pt] (N5) -- (N8);
\draw [-{Stealth[length=2mm, width=1mm]}, line width = 0.7pt] (N8) -- (N7);
\draw [-{Stealth[length=2mm, width=1mm]}, line width = 0.7pt] (N7) -- (N6);
\draw [-{Stealth[length=2mm, width=1mm]}, line width = 0.7pt] (N6) -- (N9);
\draw [-{Stealth[length=2mm, width=1mm]}, line width = 0.7pt] (N9) -- (N3);
\draw [-{Stealth[length=2mm, width=1mm]}, line width = 0.7pt] (N3) -- (N2);
\draw [-{Stealth[length=2mm, width=1mm]}, line width = 0.7pt] (N2) -- (N1);

\end{scope}

\begin{scope}[shift = {(5,0)}]
\equic[1.5cm]{9};
\draw [{Stealth[length=2mm, width=1mm]}-{Stealth[length=2mm, width=1mm]}, line width = 0.7pt] (N1) -- (N3);
\draw [{Stealth[length=2mm, width=1mm]}-{Stealth[length=2mm, width=1mm]}, line width = 0.7pt] (N2) -- (N4);
\draw [{Stealth[length=2mm, width=1mm]}-{Stealth[length=2mm, width=1mm]}, line width = 0.7pt] (N6) -- (N8);
\draw [{Stealth[length=2mm, width=1mm]}-{Stealth[length=2mm, width=1mm]}, line width = 0.7pt] (N7) -- (N9);
\path (N5) edge [loop right,line width = 0.7pt] (N5);
\end{scope}

\begin{scope}[shift = {(10, 0)}]
\equic[1.5cm]{9};
\draw [-{Stealth[length=2mm, width=1mm]}, line width = 0.7pt] (N8) -- (N9);
\draw [-{Stealth[length=2mm, width=1mm]}, line width = 0.7pt] (N9) -- (N6);
\draw [-{Stealth[length=2mm, width=1mm]}, line width = 0.7pt] (N6) -- (N7);
\draw [-{Stealth[length=2mm, width=1mm]}, line width = 0.7pt] (N7) -- (N8);

\draw [-{Stealth[length=2mm, width=1mm]}, line width = 0.7pt] (N1) -- (N2);
\draw [-{Stealth[length=2mm, width=1mm]}, line width = 0.7pt] (N2) -- (N3);
\draw [-{Stealth[length=2mm, width=1mm]}, line width = 0.7pt] (N3) -- (N4);
\draw [-{Stealth[length=2mm, width=1mm]}, line width = 0.7pt] (N4) -- (N1);

\path (N5) edge [loop right,line width = 0.7pt] (N5);
\end{scope}

\end{tikzpicture}
    \caption{The permutation $w = (1\enspace 4 \enspace 5 \enspace 8 \enspace 7 \enspace 6 \enspace 9 \enspace 3 \enspace 2) \in \mathfrak{S}_9$ (left) which clearly has non-crossing projection $(1 \enspace 2 \enspace 3 \enspace 4 \enspace 5 \enspace 6 \enspace 7 \enspace 8 \enspace 9)$, its image under $\Popt$ given by $\Popt(w)= (1 \enspace 3) (2 \enspace 4)(6 \enspace 8)(7 \enspace 9)(5)$ (middle) and $\pi_T(\Popt(w)) = (1 \enspace 2 \enspace 3 \enspace 4)(6 \enspace 7 \enspace 8 \enspace 9)(5)$ (right)}
    \label{fig:A-cex}
\end{figure}
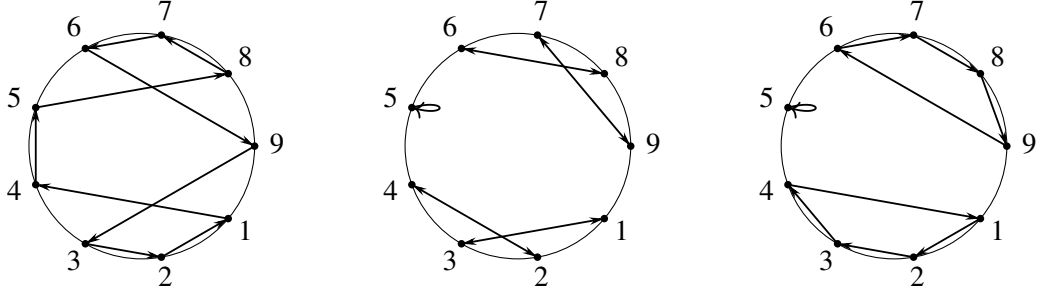

Before we proceed, let us make the following observation about $A$-splittable $w$. Firstly, $w$ has at most two cycles. If $w$ has two cycles, then it is of the form 
\[ ( n \enspace c_1 \enspace \cdots \enspace c_{\ell})(d_1 \enspace d_2 \enspace \cdots \enspace d_{n-\ell-1})\]
where $n > c_1 > \cdots > c_{\ell}$, $d_1 > c_{\ell}$ and $d_1 > d_2 > \cdots > d_{n - \ell - 1}$. This is because when $w$ has two cycles then the largest element of each of the cycles is not an antiexceedance, and so every other element has to be an antiexceedance. In particular, once we decide the two sets of elements that lie in each cycle, then each of the elements must be arranged in a decreasing order on each cycle. If $w$ has exactly one cycle with every element being an antiexceedance except for $j$ and $n$, then $w$ is of the form 
\[ (n \enspace n-1 \enspace \cdots \enspace j+1 \enspace c_1 \enspace \cdots \enspace c_{\ell} \enspace j \enspace d_1 \enspace \cdots \enspace d_{n-j-\ell - 1})\]
where $n>j > c_1 > c_2 > \cdots > c_{\ell}$ and $d_1 > d_2 > \cdots > d_{n - j - \ell - 1}$. This is because besides $n$ and $j$, every other element must be arranged in a decreasing order on the cycle. 

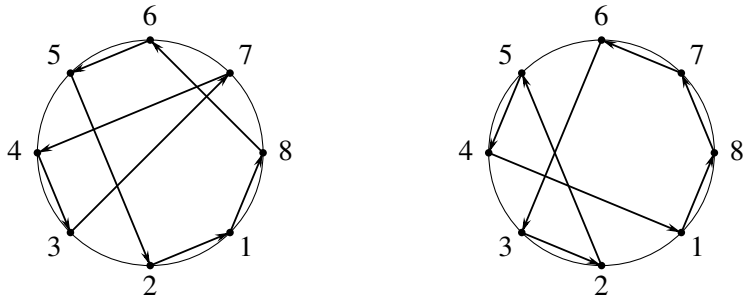
\begin{figure}[!htp]
    \centering
    \begin{tikzpicture}
\begin{scope}
\equic[1.5cm]{8};
\draw [-{Stealth[length=2mm, width=1mm]}, line width = 0.7pt] (N8) -- (N6);
\draw [-{Stealth[length=2mm, width=1mm]}, line width = 0.7pt] (N6) -- (N5);
\draw [-{Stealth[length=2mm, width=1mm]}, line width = 0.7pt] (N5) -- (N2);
\draw [-{Stealth[length=2mm, width=1mm]}, line width = 0.7pt] (N2) -- (N1);
\draw [-{Stealth[length=2mm, width=1mm]}, line width = 0.7pt] (N1) -- (N8);

\draw [-{Stealth[length=2mm, width=1mm]}, line width = 0.7pt] (N7) -- (N4);
\draw [-{Stealth[length=2mm, width=1mm]}, line width = 0.7pt] (N4) -- (N3);
\draw [-{Stealth[length=2mm, width=1mm]}, line width = 0.7pt] (N3) -- (N7);
\end{scope}

\begin{scope}[shift = {(6,0)}]
\equic[1.5cm]{8};
\draw [-{Stealth[length=2mm, width=1mm]}, line width = 0.7pt] (N8) -- (N7);
\draw [-{Stealth[length=2mm, width=1mm]}, line width = 0.7pt] (N7) -- (N6);
\draw [-{Stealth[length=2mm, width=1mm]}, line width = 0.7pt] (N6) -- (N3);
\draw [-{Stealth[length=2mm, width=1mm]}, line width = 0.7pt] (N3) -- (N2);
\draw [-{Stealth[length=2mm, width=1mm]}, line width = 0.7pt] (N2) -- (N5);
\draw [-{Stealth[length=2mm, width=1mm]}, line width = 0.7pt] (N5) -- (N4);
\draw [-{Stealth[length=2mm, width=1mm]}, line width = 0.7pt] (N4) -- (N1);
\draw [-{Stealth[length=2mm, width=1mm]}, line width = 0.7pt] (N1) -- (N8);
\end{scope}
\end{tikzpicture}
    \caption{Examples of permutations in $\mathfrak{S}_8$ with 7 antiexceedances; the permutation on the left is of the form with two cycles while that on the right has one cycle.}
    \label{fig:A-minus2aexc}
\end{figure}

The following claim allows us to extract properties of antiexceedances of the permutations in $\Popt$ orbits of $A$-splittable $w$. 

\begin{lemma}\label{claim:aexc-split}
Let $w\in \mathfrak{S}_n$ be $A$-splittable. Suppose $\cyc(\Popt^j(w)) = 1$ for all $j < i$. If neither $\alpha$ nor $\beta$ are  antiexceedances in $\Popt^i(w)$, and $\gamma > \max \{ \alpha, \beta \}$, then $\gamma$ not an antiexceedance in $\Popt^i(w)$.
\end{lemma}

It will eventually turn out that the hypothesis in Lemma~\ref{claim:aexc-split} that $\cyc(\Popt^j(w)) = 1$ for all $j < i$ is unnecessary, but a priori we do not know that $A$-splittable permutations satisfy $\cyc(\Popt^k(w)) = 1$ for all $k$.

\begin{proof}
We will use the following fact extracted from \cite[Theorem 5.2]{DN21}: $r \in \Aexc(u) \setminus \Aexc(\Popt(u))$ if and only if $u^{-1}(r)$ is the largest entry in a non-singleton cycle of $\pi_T(u)$. Let $t_k$ be the largest element in the non-singleton cycle of $\Popt^k(w)$.

Suppose $w$ consists of two cycles. If $\alpha$ and $\beta$ are in the same cycle, without loss of generality $\alpha < \beta$, so $w^{-1}(\alpha) < \alpha < \gamma < w^{-1}(\gamma)$. Note that since $w_k^{-1}(x) = v_k(v_{k-1} (\cdots v_1(w^{-1}(x)) \cdots))$ and each $v_k$ results in a clockwise movement of the indices, we get a decreasing sequence $t_i - w_i(x)$ until it reaches 0 at which point $x$ ceases to be an antiexceedance. Since $w^{-1}(\gamma) > w^{-1}(\alpha)$, it follows that the indices $j_{\alpha}, j_{\gamma}$ for which $w_i^{-1}(j_{\gamma}) = t_k$ for some $k < i$ satisfy the property that $j_{\gamma} < j_{\alpha}$. In particular, $\gamma \in \Aexc(\Popt^{j_{\gamma}}(w)) \setminus \Aexc(\Popt^{j_{\gamma+1}}(w)) \subset \Aexc(\Popt^{j_{\alpha}}(w)) \setminus \Aexc(\Popt^{j_{\alpha+1}}(w))$ which is equivalent to the desired conclusion. If $\alpha$ and $\beta$ are not in the same cycle, let $\mathscr{C}$ be the cycle in $w$ containing $\gamma$, then note that one of $\alpha$ or $\beta$, without loss of generality assume it is $\alpha$, is situated counterclockwise along $\mathscr{C}$ from $\gamma > \max\{ \alpha, \beta \}$. This implies that $w^{-1}(\alpha) < w^{-1}(\gamma)$ and we can conclude with the same argument as earlier.  

Suppose $w$ consists of one cycle, and the two elements that are \emph{not} antiexceedances are given by $j$ and $n$. As before, it suffices to prove that $w^{-1}(\gamma) > \min \{ w^{-1}(\alpha), w^{-1}(\gamma)\}$. First, note that if $\gamma > j$ then we get the desired conclusion. If $\gamma <j$, we first consider the case that $\alpha = d_{j}$ and $\beta = d_{k}$ for some indices $j$ and $k$. without loss of generality $j < k$. Then $w^{-1}(\alpha) < \beta < \gamma < w^{-1}(\gamma)$, as desired. The case that $\alpha = c_j$ and $\beta = c_{k}$ is similar. Lastly, suppose $\alpha = c_j$ and $\beta = d_k$ for some indices $j,k$. If $\gamma = c_m$ for some $m < j$ then $w^{-1}(\gamma) > w^{-1} (\alpha) $; similarly, if $\gamma = d_q$ for some $q < k$ then $w^{-1}(\gamma) > w^{-1}(\beta)$.
\end{proof}

\begin{proof}[Proof of first part of Theorem~\ref{thm:type-A}]
We have already established the ``only if'' part of the claim earlier. It suffices to check the ``if'' direction. 

We induct on $i$ to show that if $w \in \mathfrak{S}_n$ is a $A$-splittable permutation then $\cyc(\pi_T(\Popt^i(w))) = 1$. The base case $i = 0$ is true by assumption. Next, we handle the inductive step. Suppose the largest element in the non-singleton cycle of $\pi_T(\Popt^i(w))$ is $t$ and the smallest element is $\tilde{t}$. For simplicity of notation, write $\Popt^i(w) = w_i$ and $\pi_T(\Popt^i(w)) = v_i$. 

Now, we must prove that $\cyc(\pi_T(\Popt^{i+1}(w))) = 1$. Equivalently, we need to prove that for any $r > s$ with $r,s \in \{x: w_{i+1}(x) \neq x\}$ and $s \neq v_{i+1}(r)$ (that is, $r,s$ are not fixed points under $w_{i+1}$ and $s$ does not immediately succeed $r$ in $v_{i+1}$) one of the following conditions holds:
\begin{enumerate}
    \item [(1)] There exist $a \in [s+1, r-1]$ and $b \in [r+1,n] \cup [1,s-1]$ such that either $w_{i+1}(a) = b$ or $w_{i+1}(b) = a$. 
    
    \begin{figure}[!htp]
        \centering
        \begin{tikzpicture}
\begin{scope}
\equicnl[1.2cm]{8}
\draw[-, line width = 0.7pt, dashed] (N1) -- (N4);
\draw[-{Stealth[length=2mm, width=1mm]}, line width = 0.7pt] (N2) -- (N6);
\draw (N1) node[right] {$r$};
\draw (N4) node[left] {$s$};
\draw (N2) node[below] {$a$};
\draw (N6) node[above] {$b$};
\end{scope}

\begin{scope}[shift={(5,0)}]
\equicnl[1.2cm]{8}
\draw[-, line width = 0.7pt, dashed] (N1) -- (N4);
\draw[-{Stealth[length=2mm, width=1mm]}, line width = 0.7pt] (N6) -- (N2);
\draw (N1) node[right] {$r$};
\draw (N4) node[left] {$s$};
\draw (N2) node[below] {$a$};
\draw (N6) node[above] {$b$};
\end{scope}
\end{tikzpicture}
        \caption{Illustration of condition (1).}
        \label{fig:A-cond1}
    \end{figure}
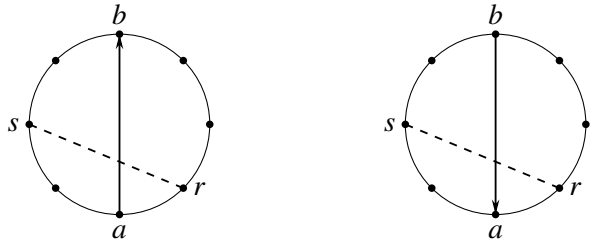
\item [(2)] There exist $e_1 \in [r+1, s-1]$ and $f_1 \in [s+1, n] \cup [1,r-1]$ such that either $w_{i+1}(e_1) = r$ and $w_{i+1}(r) = f_1$ or $w_{i+1}(f_i) = r$ and $w_{i+1}(r) = e_1$. 

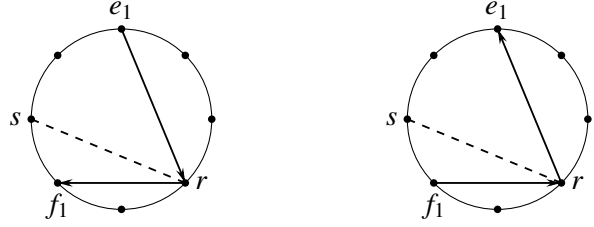
\begin{figure}[!htp]
    \centering
\begin{tikzpicture}
\begin{scope}
\equicnl[1.2cm]{8}
\draw[-, line width = 0.7pt, dashed] (N1) -- (N4);
\draw[-{Stealth[length=2mm, width=1mm]}, line width = 0.7pt] (N1) -- (N3);
\draw[-{Stealth[length=2mm, width=1mm]}, line width = 0.7pt] (N6) -- (N1);
\draw (N1) node[right] {$r$};
\draw (N4) node[left] {$s$};
\draw (N3) node[below] {$f_1$};
\draw (N6) node[above] {$e_1$};
\end{scope}

\begin{scope}[shift={(5,0)}]
\equicnl[1.2cm]{8}
\draw[-, line width = 0.7pt, dashed] (N1) -- (N4);
\draw[-{Stealth[length=2mm, width=1mm]}, line width = 0.7pt] (N1) -- (N6);
\draw[-{Stealth[length=2mm, width=1mm]}, line width = 0.7pt] (N3) -- (N1);
\draw (N1) node[right] {$r$};
\draw (N4) node[left] {$s$};
\draw (N3) node[below] {$f_1$};
\draw (N6) node[above] {$e_1$};
\end{scope}
\end{tikzpicture}
    \caption{Illustration of condition (2).}
    \label{fig:A-cond2}
\end{figure}
\item [(3)] There exist $e_1 \in [r+1, s-1]$ and $f_1 \in [s+1, n] \cup [1,r-1]$ such that either $w_{i+1}(e_1) = s$ and $w_{i+1}(s) = f_1$ or $w_{i+1}(f_i) = s$ and $w_{i+1}(s) = e_1$. 

\begin{figure}[!htp]
    \centering
    \begin{tikzpicture}
\begin{scope}
\equicnl[1.2cm]{8}
\draw[-, line width = 0.7pt, dashed] (N1) -- (N4);
\draw[-{Stealth[length=2mm, width=1mm]}, line width = 0.7pt] (N4) -- (N2);
\draw[-{Stealth[length=2mm, width=1mm]}, line width = 0.7pt] (N7) -- (N4);
\draw (N1) node[right] {$r$};
\draw (N4) node[left] {$s$};
\draw (N2) node[below] {$f_1$};
\draw (N7) node[above] {$e_1$};
\end{scope}

\begin{scope}[shift={(5,0)}]
\equicnl[1.2cm]{8}
\draw[-, line width = 0.7pt, dashed] (N1) -- (N4);
\draw[-{Stealth[length=2mm, width=1mm]}, line width = 0.7pt] (N4) -- (N7);
\draw[-{Stealth[length=2mm, width=1mm]}, line width = 0.7pt] (N2) -- (N4);
\draw (N1) node[right] {$r$};
\draw (N4) node[left] {$s$};
\draw (N2) node[below] {$f_1$};
\draw (N7) node[above] {$e_1$};
\end{scope}
\end{tikzpicture}
    \caption{Illustration of condition (3).}
    \label{fig:A-cond3}
\end{figure}
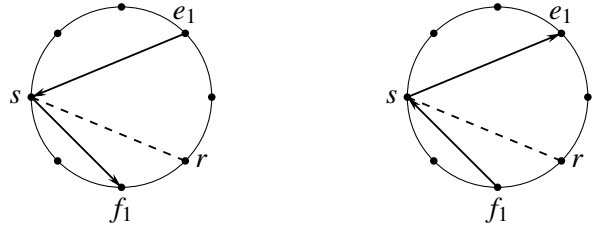
\item [(4)] We have $w_{i+1}(r) = s$, $\{ w_{i+1}^{-1}(r), w_{i+1}(s)\} \cap [r+1, s-1] \neq \emptyset$ and $\{ w_{i+1}^{-1}(r), w_{i+1}(s)\} \cap \{ [s+1,n] \cup [1,r-1]\}$.
\begin{figure}[!htp]
    \centering
    \begin{tikzpicture}
\begin{scope}
\equicnl[1.2cm]{8}
\draw[-{Stealth[length=2mm, width=1mm]}, line width = 0.7pt] (N1) -- (N4);
\draw[-{Stealth[length=2mm, width=1mm]}, line width = 0.7pt] (N3) -- (N1);
\draw[-{Stealth[length=2mm, width=1mm]}, line width = 0.7pt] (N4) -- (N7);
\draw (N1) node[right] {$r$};
\draw (N4) node[left] {$s$};
\draw (N7) node[above, shift={(0.5,0)}] {$w_{i+1}(s)$};
\draw (N3) node[below,shift={(-0.5,0)}] {$w_{i+1}^{-1}(r)$};
\end{scope}

\begin{scope}[shift={(5,0)}]
\equicnl[1.2cm]{8}
\draw[-{Stealth[length=2mm, width=1mm]}, line width = 0.7pt] (N1) -- (N4);
\draw[-{Stealth[length=2mm, width=1mm]}, line width = 0.7pt] (N4) -- (N3);
\draw[-{Stealth[length=2mm, width=1mm]}, line width = 0.7pt] (N7) -- (N1);
\draw (N1) node[right] {$r$};
\draw (N4) node[left] {$s$};
\draw (N7) node[above, shift={(0.5,0)}] {$w_{i+1}^{-1}(r)$};
\draw (N3) node[below,shift={(-0.5,0)}] {$w_{i+1}(s)$};
\end{scope}
\end{tikzpicture}
    \caption{Illustration of condition (4).}
    \label{fig:A-cond4}
\end{figure}
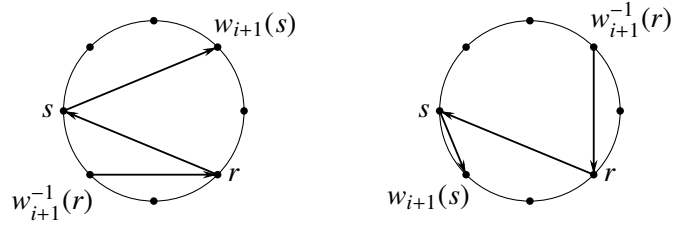
\item [(5)] We have $w_{i+1}(s) = r$, $\{ w_{i+1}^{-1}(s), w_{i+1}(r)\} \cap [r+1, s-1] \neq \emptyset$ and $\{ w_{i+1}^{-1}(s), w_{i+1}(r)\} \cap \{ [s+1,n] \cup [1,r-1]\}$.
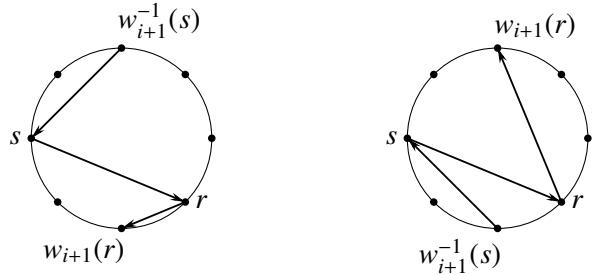
\begin{figure}[!htp]
    \centering
    \begin{tikzpicture}
\begin{scope}
\equicnl[1.2cm]{8}
\draw[-{Stealth[length=2mm, width=1mm]}, line width = 0.7pt] (N4) -- (N1);
\draw[-{Stealth[length=2mm, width=1mm]}, line width = 0.7pt] (N1) -- (N2);
\draw[-{Stealth[length=2mm, width=1mm]}, line width = 0.7pt] (N6) -- (N4);
\draw (N1) node[right] {$r$};
\draw (N4) node[left] {$s$};
\draw (N6) node[above, shift={(0.5,0)}] {$w_{i+1}^{-1}(s)$};
\draw (N2) node[below,shift={(-0.5,0)}] {$w_{i+1}(r)$};
\end{scope}

\begin{scope}[shift={(5,0)}]
\equicnl[1.2cm]{8}
\draw[-{Stealth[length=2mm, width=1mm]}, line width = 0.7pt] (N4) -- (N1);
\draw[-{Stealth[length=2mm, width=1mm]}, line width = 0.7pt] (N2) -- (N4);
\draw[-{Stealth[length=2mm, width=1mm]}, line width = 0.7pt] (N1) -- (N6);
\draw (N1) node[right] {$r$};
\draw (N4) node[left] {$s$};
\draw (N6) node[above, shift={(0.5,0)}] {$w_{i+1}(r)$};
\draw (N2) node[below,shift={(-0.5,0)}] {$w_{i+1}^{-1}(s)$};
\end{scope}
\end{tikzpicture}
    \caption{Illustration of condition (5).}
    \label{fig:A-cond5}
\end{figure}
\end{enumerate}

Since $\{ x: w_{i+1}(x) \neq x \} \subset \{ x: w_i(x) \neq x \}$, note that coupled with the induction hypothesis we know that for any $r, s$ satisfying the aformentioned properties, one of the above five conditions hold when we condition the permutation $w_i$ in lieu of $w_{i+1}$. In particular, we will show that after operating with the $\Popt$ operator, we preserve the property that one of the above five conditions continues to hold. 

Before we prove that claim, we first make a reduction. Note that we may without loss of generality assume that $w_i(r) \in [r+1, s]$ because if $w_i(r) \in [s+1, n] \cup [1, r-1]$, then since $v_i(r) \in [r+1, s-1]$, it follows that we have $w_{i+1}(v_i(r)) = w_i(r)$ satisfies condition (1) above. In particular, this in turn implies that we may without loss of generality assume that $w_i(r)$ is not an antiexceedance in $w_i$.

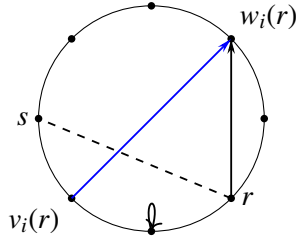
\begin{figure}[!htp]
    \centering
    \begin{tikzpicture}
\equicnl[1.5]{8};
\draw[-, line width = 0.7pt, dashed] (N1) -- (N4);
\draw[-{Stealth[length=2mm, width=1mm]}, line width = 0.7pt] (N1) -- (N7);
\draw (N1) node[right] {$r$};
\draw (N4) node[left] {$s$};
\draw (N7) node[above,shift={(0.5,0)}] {$w_{i}(r)$};
\draw (N3) node[below,shift={(-0.5,0)}] {$v_{i}(r)$};
\path (N2) edge [loop above,line width = 0.7pt] (N2);
\draw[-{Stealth[length=2mm, width=1mm]}, line width = 0.7pt, blue] (N3) -- (N7);
\end{tikzpicture}
    \caption{Illustration of why if $w_i(r) \in [s+1, n] \cup [1, r-1]$, then condition (1) is satisfied. The black edge is part of the diagram for $w_i$ and the blue edge is part of the diagram for $w_{i+1}$.}
    \label{fig:my_label}
\end{figure}

An identical argument shows that we may suppose without loss of generality that $w_i(s) \in [s+1, n] \cup [1,r]$. 

Now, we split into cases. 
\begin{enumerate}
    \item [(I)] Condition (1) above is satisfied for $w_i$ and there exists $a \in [r+1, s-1]$ and $b \in [s+1,n] \cup [1,r-1]$ such that $w_{i+1}(b) = a$. Note that if $b \neq v_i^{-1}(r)$, then $w_{i+1}^{-1}(b) = v_i(w_i^{-1}(b)) = v_i(a) \in [s+1,n] \cup [1,r-1]$, and consequently condition (1) would continue to hold for $w_{i+1}$. It remains to handle the case that $b = v_i^{-1}(r)$. 
    
    Suppose for the sake of contradiction that $w_{i+1}$ does not satisfy one of the five conditions.
    
    First, consider the situation of $v_i^{-1}(r) < r$.  Then, because $w_i(v_i^{-1}(r)) \in [r+1,s]$, it is not an antiexceedance in $w_i$. Now, our reduction above gave us that $w_i(r)$ is not an antiexceedance in $w_i$.
    
    We claim that $w_i^{-1}(t) \not \in [1, r-1]$. To see why this is the case, observe that because $\cyc(\pi_T(v_j)) = 1$ for all $j \leq i$, it follows that if $x,y,z$ are not fixed points of $w_i$ (and consequently not fixed points of $w_j$ for any $j \leq i$) then if $w_i^{-1}(x), w_i^{-1}(y), w_i^{-1}(z)$ appear in this order clockwise on the circle then $w_j^{-1}(x), w_j^{-1}(y), w_j^{-1}(z)$ appear in this order clockwise for all $j \leq i$. If $w_i^{-1}(t) \in [1, r-1]$, then in particular we have that $w_i^{-1}(t), v_i^{-1}(r), r$ appear clockwise in this order. In particular, there exists some $k$ such that $w_k^{-1}(t) > t, w_k^{-1}(w_i(v_i^{-1}(r))), w_k^{-1}(w_i(r))$. In $w_k$, $t$ is an antiexceedance while $v_i^{-1}(r), r$ are not antiexceedances. This is a contradiction to Lemma~\ref{claim:aexc-split}.
    
    Let $\mathcal{N}$ be the set of numbers clockwise from $v_i^{-1}(s)$ to $t$ lying in $v_i$ that are not antiexceedances in $w_i$. By Lemma~\ref{claim:aexc-split}, all numbers in $\mathcal{N}$ are not antiexceedances in $w_i$. Using an identical argument as the previous paragraph, we can show that any $q \in \mathcal{N}$ satisfies $w_i^{-1}(q) \not \in [1,r-1]$.
    
    Let $\mathcal{Z}$ denote the set of elements situated clockwise from $\tilde{t}$ to $r$. We claim that there exists an element $z \in \mathcal{Z}$ such that $w_i^{-1}(z) \in [s, n]$. This would then provide a contradiction to the previous paragraph. To show the claim, first note that if $w_i^{-1}(z) \in [r+1, s-1]$ then by the previous paragraph it follows that $w_i^{-1}(z) < v_i^{-1}(s)$ and in particular $w_{i+1}^{-1}(z) \in [r+1, s-1]$ as well which would then allow us to satisfy condition (1). Consequently, it suffices to rule out the possibility that $w_i^{-1}(z) \in [1, s-1]$ for all such elements $z \in \mathcal{Z}$. Note that $w_i^{-1}(z) \neq v_i^{-1}(s)$. Consequently, there are only $|\mathcal{Z}| - 1$ possibilities for the values of $w_i^{-1}(z)$, which is a contradiction. Refer to Figure~\ref{fig:A-case(I)} for a depiction of this claim.
    \begin{figure}[!htp]
        \centering
        \begin{tikzpicture}
\begin{scope}
\equicnl[2.5]{16};
\draw[-, dashed, line width = 0.7pt] (N1) -- (N9);
\draw (N1) node[right] {$r$};
\draw (N9) node[left] {$s$};
\path (N16) edge [loop left,line width = 0.7pt] (N16);
\draw[-{Stealth[length=2mm, width=1mm]}, line width = 0.7pt] (N15) -- (N4);
\draw (N15) node[right] {$v_i^{-1}(r)$};
\draw (N4) node[below, shift={(-0.5,0)}] {$w_i(v_i^{-1}(r))$};
\fill[red](N4) circle (0.07 cm);
\draw[-{Stealth[length=2mm, width=1mm]}, line width = 0.7pt] (N1) -- (N7);
\fill[red](N7) circle (0.07 cm);
\draw (N7) node[left] {$w_i(r)$};
\draw (N11) node[above] {$t$};
\path (N12) edge [loop below,line width = 0.7pt] (N12);
\draw (N13) node[above] {$\tilde{t}$};
\draw (N14) node[above] {$z$};
\draw[-{Stealth[length=2mm, width=1mm]}, line width = 0.7pt] (N14) -- (N10);
\end{scope}

\begin{scope}[shift={(7,0)}]
\equicnl[2.5]{15};
\draw[-, dashed, line width = 0.7pt] (N1) -- (N8);
\draw (N1) node[right] {$\tilde{t}$};
\draw (N8) node[left] {$s$};
\path (N15) edge [loop left,line width = 0.7pt] (N15);
\path (N14) edge [loop left,line width = 0.7pt] (N14);
\draw (N13) node[right, shift = {(0.1,0)}] {$t$};
\draw[-{Stealth[length=2mm, width=1mm]}, line width = 0.7pt] (N1) -- (N5);
\draw[-{Stealth[length=2mm, width=1mm]}, line width = 0.7pt] (N13) -- (N4);
\fill[red](N4) circle (0.07 cm);
\fill[red](N5) circle (0.07 cm);
\draw (N4) node[below, shift = {(0.1,0)}] {$w_i(t)$};
\draw (N5) node[left, shift = {(-0.1,-0.1)}] {$w_i(\tilde{t})$};
\draw[-{Stealth[length=2mm, width=1mm]}, blue, line width = 0.7pt] (N2) -- (N5);
\draw[-{Stealth[length=2mm, width=1mm]}, blue, line width = 0.7pt] (N1) -- (N4);
\path (N9) edge [loop right,line width = 0.7pt] (N9);

\draw[-{Stealth[length=2mm, width=1mm]}, blue, dotted, line width = 0.7pt] (N10) -- (N6);
\draw[-{Stealth[length=2mm, width=1mm]}, blue, dashed, line width = 0.7pt] (N10) -- (N8);
\draw[-{Stealth[length=2mm, width=1mm]}, blue, dashed, line width = 0.7pt] (N8) -- (N7);
\draw (N10) node[above,shift={(0.1,0.1)}] {$v_i^{-1}(s)$};
\end{scope}
\end{tikzpicture}
        \caption{Illustrations for case (I). The figure on the left is a depiction of the situation when $v_i^{-1}(r) < r$. On the right, the figure shows the situation when $r = \tilde{t}$. Here the blue edges should be thought of as part of the diagram for $w_{i+1}$. The red nodes are not antiexceedances.}
        \label{fig:A-case(I)}
    \end{figure}
    Next, we handle the case where $v_i^{-1}(r) > r$. This case can only occur if $ r = \tilde{t}$ and $v_i^{-1}(r) = t$. Note by our earlier reduction we have that $w_i(\tilde{t}) \geq w_{i+1}^{-1}(w_i(\tilde{t}))$ and since $w_{i+1}(r) = w_i(t) > r$, it follows that $w_i(t)$ and $w_i(\tilde{t})$ are both not antiexceedances in $w_{i+1}$. By Lemma~\ref{claim:aexc-split}, we have that $s$ and $v_i(s)$ are both not antiexceedances in $w_{i+1}$ as well. If $w_{i+1}^{-1}(v_i(s)) \in [\tilde{t}, s-1]$ then we satisfy condition (1). Otherwise, we necessarily have $w_{i+1}^{-1}(v_i(s)) = s$. But we also know that $w_{i+1}^{s} < s$, and consequently we satisfy condition (3), a contradiction as desired.

    
    \item [(II)] Condition (1) above is satisfied for $w_i$ and there exists $a \in [r+1, s-1]$ and $b \in [s+1,n] \cup [1,r-1]$ such that $w_{i+1}(a) = b$. We may assume that $w_i(v_i^{-1}(r)) \not \in [r+1,s-1]$ as otherwise we may reduce to case (I). Analogously as above, if $a \neq v_i^{-1}(s)$ then we get that condition (1) is satisfied for $w_{i+1}$. Consequently, henceforth we will assume that $a = v_i^{-1}(s)$.
    
    Suppose for the sake of contradiction that $w_{i+1}$ does not satisfy one of the five conditions.
    
    If $w_i^{-1}(v_i^{-1}(s)) > v_i^{-1}(s)$ then it follows that $w_i^{-1}(v_i^{-1}(s)) \in [s+1,n] \cup [1, v_{i}^{-1}(r))$. Here we have $w_i^{-1}(v_i^{-1}(s)) \neq s$ because $s$ is not a fixed point in $w_{i+1}$ and $w_i^{-1}(v_i^{-1}(s)) \neq v_i^{-1}(r)$ because we are not in case (I). In particular, $w_{i+1}^{-1}(v_i^{-1}(s)) \in [s+1,n]\cup[1,r-1]$ and so condition (1) is satisfied for $w_{i+1}$, which is a contradiction. Now, we have that $v_i^{-1}(s) \in [r+1, s-1]$ is not an antiexceedance in $w_i$ and so in particular $v_i^{-2}(s) \in [r+1, s-1]$. Similarly, since $v_i^{-2}(s) \neq v_i^{-1}(r)$ we must have that $v_i^{-2}(s) \in [r+1,s-1]$ and it is not an antiexceedance in $w_i$ either. By Lemma~\ref{claim:aexc-split} we have that $s$ is not an antiexceedance in $w_i$ as well. We can show that $w_i^{-1}(s) \not \in [1,r-1]$ by a similar argument as in case (I) by considering the relative clockwise positions of $w_i^{-1}(s), w_i^{-1}(v_i^{-1}(s)), w_i^{-1}(v_i^{-2}(s))$. Consequently, $w_i^{-1}(s) \in [r+1,s-1]$ and in particular $w_{i+1}(s) \in [s+1,n]\cup[1,r-1]$ while $w_{i+1}^{-1}(s) \in [r+1,s-1]$ so that condition (3) is satisfied for $w_{i+1}$, which is a contradiction.
    
    \begin{figure}[!htp]
        \centering
        \begin{tikzpicture}
\equicnl[2.5]{15};
\draw[-, dashed, line width = 0.7pt] (N1) -- (N9);
\draw (N1) node[right] {$r$};
\draw (N9) node[left] {$s$};
\draw[-{Stealth[length=2mm, width=1mm]}, line width = 0.7pt] (N2) -- (N5);
\draw (N5) node[below,shift={(-0.1,-0.1)}] {$v_i^{-2}(s)$};
\draw[-{Stealth[length=2mm, width=1mm]}, line width = 0.7pt] (N3) -- (N9);
\draw[-{Stealth[length=2mm, width=1mm]}, line width = 0.7pt] (N4) -- (N8);
\path (N7) edge [loop right,line width = 0.7pt] (N7);
\path (N6) edge [loop right,line width = 0.7pt] (N6);
\draw (N8) node[left] {$v_i^{-1}(s)$};
\draw[-{Stealth[length=2mm, width=1mm]}, line width = 0.7pt] (N8) -- (N12);
\draw[-{Stealth[length=2mm, width=1mm]}, blue, line width = 0.7pt] (N9) -- (N12);
\draw[-{Stealth[length=2mm, width=1mm]}, blue, line width = 0.7pt] (N4) -- (N9);
\fill[red](N5) circle (0.07 cm);
\fill[red](N8) circle (0.07 cm);

\end{tikzpicture}
        \caption{An illustration of case (II). Here the red nodes are not antiexceedances. The blue edges should be thought of as part of $w_{i+1}$.}
        \label{fig:my_label}
    \end{figure}
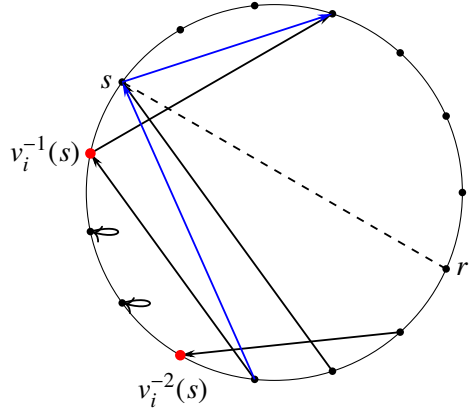
    \item [(III)] Condition (2) above is satisfied for $w_i$. Because of the reduction we made before we split into cases, in this case we necessarily have that $w_i^{-1}(r) \in [s+1,n] \cup [1,r-1]$ and $w_i(r) \in [r+1,s-1]$. Suppose for the sake of contradiction that $w_{i+1}$ does not satisfy one of the five conditions. We may assume that we are not in either of cases (I) or (II). That is, we may assume that $w_i(v_i^{-1}(r)) \in [s+1, n] \cup [1,r]$ and
    \begin{equation}\label{eq:1}
        w_i(v_i^{-1}(s)) \in [r, s-1].
    \end{equation}
   Note that since by assumption of condition (2), $w_i^{-1}(r) \neq v_i^{-1}(s)$ so that $w_{i+1}(s) \in [r+1, s-1]$. Now we split further into cases. First, if $w_i^{-1}(s) \in [s+1,n] \cup [1, v_i^{-1}(r))$, then we satisfy condition (2), a contradiction. If $w_i^{-1}(s) = v_i^{-1}(r)$ so that $w_{i+1}(r) = s$ then (\ref{eq:1}) tells us that condition (4) is satisfied, a contradiction. Otherwise, $w_i^{-1}(s) < s$ and both $w_i(r)$ and $s$ are not antiexceedances in $w_i $ and so it follows from Lemma~\ref{claim:aexc-split} that $v_i(s)$ is also not an antiexceedance in $w_i$. As in case (I), we may argue that $w_i^{-1}(v_i(s)) \not \in [1, r-1]$ by considering the relative clockwise positions of $ w_i^{-1}(v_i(s)), r, w_i^{-1}(v_i^{-1}(s))$. Consequently, $w_i^{-1}(v_i(s)) \in [r+1,s]$. 
    
    Suppose $w_i^{-1}(v_i(s)) \neq s$. By assumption since we are not in case (II), it follows that $w_i^{-1}(v_i(s)) \neq v_i^{-1}(s)$ so that $w_{i+1}(v_i(s)) \in [1,s-1]$ and so condition (1) is satisfied, which is a contradiction. 
    
    It remains to handle the case where $w_i^{-1}(v_i(s)) = s$. Now, consider the same argument as above for $v_i^2(s)$. Similarly, we show that if $w_i^{-1}(v_i^2(s)) \neq v_i(s)$, then $w_{i+1}$ satisfies one of the five conditions. Iteratively applying this argument it follows that for all elements $z$ in $v_i$ clockwise from $v_i(s)$ to $t$, we have that $v_{i+1}(z) = z$. In particular, since $w_i^{-1}(r), w_i^{-1}(v_i^{-1}(r)) \in \{ t \} \cup [1,r-1]$, this implies that $ w_{i+1}^{-1}(r), w_{i+1}^{-1}(v_i^{-1}(r)) \in [\tilde{t}, v_i^{-1}(r))$ are both not antiexceedances in $w_{i+1}$. By Lemma~\ref{claim:aexc-split}, this implies that $v_i(r)$ is not an antiexceedance in $w_{i+1}$ as well, and in particular because we are not in case (I), we have that $w_{i+1}^{-1}(v_i(r)) < r$, which means $w_{i+1}$ satisfies condition (1), a contradiction. 
    
    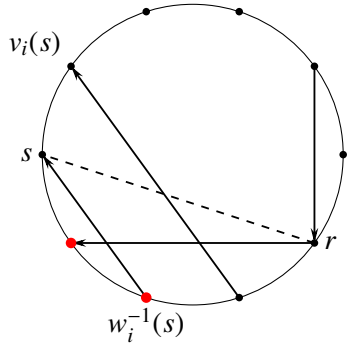
\begin{figure}[!htp]
        \centering
        \begin{tikzpicture}
\equicnl[2]{10};
\draw[-, dashed, line width = 0.7pt] (N1) -- (N5);
\draw (N1) node[right] {$r$};
\draw (N5) node[left] {$s$};
\draw[-{Stealth[length=2mm, width=1mm]}, line width = 0.7pt] (N9) -- (N1);
\draw[-{Stealth[length=2mm, width=1mm]}, line width = 0.7pt] (N1) -- (N4);
\draw[-{Stealth[length=2mm, width=1mm]}, line width = 0.7pt] (N2) -- (N6);
\draw[-{Stealth[length=2mm, width=1mm]}, line width = 0.7pt] (N3) -- (N5);
\draw (N3) node[below] {$w_i^{-1}(s)$};
\draw (N6) node[above left] {$v_i(s)$};
\fill[red](N3) circle (0.07 cm);
\fill[red](N4) circle (0.07 cm);
\end{tikzpicture}
        \caption{An illustration for case (III). The nodes labelled red are not antiexceedances.}
        \label{fig:my_label}
    \end{figure}
    
    \item [(IV)] Condition (3) is satisfied for $w_i$. This case is very similar to case (III). Because of the reduction that we made before we split into cases, in this case we necessarily have that $w_i^{-1}(s) \in [r+1, s-1]$ and $w_i(s) \in [s+1, n]\cup [1,r-1]$. Suppose for the sake of contradiction that $w_{i+1}$ does not satisfy one of the five conditions. We may assume that we are not in either of case (I) so that $w_i^{-1}(v_i^{-1}(s)) \neq v_i^{-1}(r)$ and in particular $w_i^{-1}(v_i^{-1}(s)) \in [r+1, s-1]$. That is, $v_i^{-1}(s)$ is not an antiexceedance in $w_i$. We may also assume that we are not in case (II) so that $w_i^{-1}(v_{i}(s)) \neq v_i^{-1}(r)$. Furthermore, we may also assume that $w_i^{-1}(r) \in [s+1,n] \cup [1,r-1]$. This is because $w_{i+1}(r) = w_i(v_i^{-1}(r)) \in [s+1,n] \cup [1,r-1]$ and if $w_i^{-1}(r) \in [r+1, s-1]$ then $w_{i+1}$ would satisfy condition (2) if $w_i^{-1}(r) \in [s-1]$ and would satisfy condition (4) if $w_i^{-1}(r) = s$ since $ w_i^{-1}(v_i^{-1}(s)) \in [r+1, s-1]$.

    Since $s$ is also by assumption not an antiexceedance in $w_i$, it follows by Lemma~\ref{claim:aexc-split} that $v_i(s)$ is not an antiexceedance in $w_i$ either. By a similar argument as in case (I), we can show that $w_i^{-1}(v_i(s)) \not \in [1,r-1]$ by considering the relative clockwise positions of $w_i^{-1}(v_i(s)), w_i^{-1}(s), w_i^{-1}(v_i^{-1}(s))$.
    
    If $w_i^{-1}(v_i(s)) < v_i^{-1}(s)$, then we get that $w_{i+1}$ satisfies condition (2), a contradiction. Otherwise, we have that $w_i^{-1}(v_i(s)) = s$ so that $w_{i+1}(v_i(s)) = v_i(s)$. 
    
    Now, consider the same argument for $v_i^2(s)$. Similarly, we show that if $w_i^{-1}(v_i^2(s)) \neq v_i(s)$ then $w_{i+1}$ satisfies one of the five conditions. Iteratively applying this argument it follows that for all elements $z$ in $v_i$ clockwise from $v_i(s)$ to $t$, we have that $v_{i+1}(z) = z$. In particular, since $w_i^{-1}(r), w_i^{-1}(v_i^{-1}(r)) \in \{ t \} \cup [1,r-1]$, this implies that $ w_{i+1}^{-1}(r), w_{i+1}^{-1}(v_i^{-1}(r)) \in [\tilde{t}, v_i^{-1}(r))$ are both not antiexceedances in $w_{i+1}$. By Lemma~\ref{claim:aexc-split}, this implies that $v_i(r)$ is not an antiexceedance in $w_{i+1}$ as well, and in particular because we are not in case (I), we have that $w_{i+1}^{-1}(v_i(r)) < r$, which means $w_{i+1}$ satisfies condition (1), a contradiction. 
    
    \item [(V)] Condition (4) is satisfied for $w_i$. Because of the reduction we made before splitting into cases, it follows that we necessarily have that $w_i^{-1}(r) \in [r+1, s-1]$ and $w_i(s) \in [s+1,n] \cup [1,r-1]$. We may assume that we are not in cases (I) or (II) so that $w_i^{-1}(v_i^{-1}(r)) \in [s+1,n] \cup [1,r-1]$ and $w_i^{-1}(v_i^{-1}(s)) \in [r,s-1]$. 
    If $w_i(r) \neq v_i^{-1}(s)$, then $w_{i+1}^{-1}(r) \in [r+1,s-1]$ and $w_{i+1}(r) = w_i^{-1}(v_i^{-1}(r)) \in [s+1,n] \cup [1,r-1]$ so that $w_{i+1}$ satisfies condition (1).
    
    If $w_i(r) = v_i^{-1}(s)$, then $w_{i+1}^{-1}(s) = v_i(r) \in [r+1, s-1]$, $w_{i+1}(s) = r$ and $w_{i+1}^{-1}(r) = w_i^{-1}(v_i^{-1}(r)) \in [s+1,n] \cup [1,r-1]$ so that $w_{i+1}$ satisfies condition (5). 
    
    \begin{figure}[!htp]
        \centering
        \begin{tikzpicture}
\begin{scope}
\equicnl[2]{13};
\draw[-{Stealth[length=2mm, width=1mm]}, line width = 0.7pt] (N1) -- (N7);
\draw (N1) node[right] {$r$};
\draw (N7) node[left] {$s$};
\draw[-{Stealth[length=2mm, width=1mm]}, line width = 0.7pt] (N3) -- (N1);
\draw[-{Stealth[length=2mm, width=1mm]}, line width = 0.7pt] (N7) -- (N9);
\draw[-{Stealth[length=2mm, width=1mm]}, line width = 0.7pt] (N13) -- (N11);
\draw[-{Stealth[length=2mm, width=1mm]}, blue, line width = 0.7pt] (N1) -- (N11);
\draw[-{Stealth[length=2mm, width=1mm]}, blue, line width = 0.7pt] (N4) -- (N1);
\end{scope}

\begin{scope}[shift={(6,0)}]
\equicnl[2]{15};
\draw[-{Stealth[length=2mm, width=1mm]}, line width = 0.7pt] (N1) -- (N8);
\draw (N1) node[right] {$r$};
\draw (N8) node[left] {$s$};
\draw (N3) node[below] {$v_i(r)$};
\draw (N6) node[left] {$v_i^{-1}(s)$};

\path (N2) edge [loop left,line width = 0.7pt] (N2);
\path (N7) edge [loop right,line width = 0.7pt] (N7);
\draw[-{Stealth[length=2mm, width=1mm]}, line width = 0.7pt] (N6) -- (N1);
\draw[-{Stealth[length=2mm, width=1mm]}, line width = 0.7pt] (N15) -- (N13);
\draw[-{Stealth[length=2mm, width=1mm]}, line width = 0.7pt] (N8) -- (N11);
\draw[-{Stealth[length=2mm, width=1mm]}, blue, dashed, line width = 0.7pt] (N3) -- (N8);
\draw[-{Stealth[length=2mm, width=1mm]}, blue, dashed, line width = 0.7pt] (N8) -- (N1);
\draw[-{Stealth[length=2mm, width=1mm]}, blue, dashed, line width = 0.7pt] (N1) -- (N13);
\end{scope}
\end{tikzpicture}
        \caption{An illustration for case (VI). The black edges are part of the diagram for $w_i$ while the blue edges are part of the diagam for $w_{i+1}$. On the left, $w_i(r) \neq v_i^{-1}(s)$ and on the right we have the situation where $w_i(r) = v_i^{-1}(s)$.}
        \label{fig:A-case5}
    \end{figure}
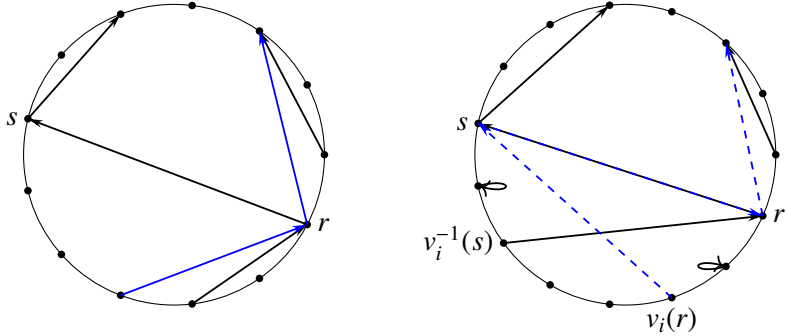
    \item [(VI)] Condition (5) is satisfied for $w_i$. Because of the reduction we made before splitting into cases, it follows that we necessarily have that $w_i(r) \in [r+1, s-1]$ and $w_i^{-1}(s) \in [s+1,n] \cup [1,r-1]$.  We may assume that we are not in case (II) so that $w_i^{-1}(v_i^{-1}(r)) \in [s,n] \cup [1,r-1]$ and $w_i^{-1}(v_i^{-1}(s)) \in [r+1,s-1]$.
    
    If $w_i^{-1}(s) \neq v_i^{-1}(r)$, then $w_{i+1}(s) \in [s+1,n] \cup [1,r-1]$ while $w_{i+1}(s) = w_i^{-1}(v_i^{-1}(s)) \in [r+1,s-1]$ so that $w_{i+1}$ satisfies condition (2). 
    
    If $w_i^{-1}(s) = v_i^{-1}(r)$, then $w_{i+1}(r) = s$, $w_{i+1}^{-1}(r) = v_i(s) \in [s+1,n] \cup [1,r-1]$ while $w_{i+1}(s) = w_i^{-1}(v_i^{-1}(s)) \in [r+1,s-1]$ so that $w_{i+1}$ satisfies condition (4).
    
    \begin{figure}
        \centering
        
\begin{tikzpicture}
\begin{scope}
\equicnl[2]{14};
\draw[-{Stealth[length=2mm, width=1mm]}, line width = 0.7pt] (N8) -- (N1);
\draw (N1) node[right] {$r$};
\draw (N8) node[left] {$s$};
\draw[-{Stealth[length=2mm, width=1mm]}, line width = 0.7pt] (N1) -- (N4);
\draw[-{Stealth[length=2mm, width=1mm]}, line width = 0.7pt] (N11) -- (N8);
\draw[-{Stealth[length=2mm, width=1mm]}, line width = 0.7pt] (N7) -- (N5);
\draw[-{Stealth[length=2mm, width=1mm]}, blue, line width = 0.7pt] (N12) -- (N8);
\draw[-{Stealth[length=2mm, width=1mm]}, blue, line width = 0.7pt] (N8) -- (N5);
\end{scope}

\begin{scope}[shift={(6,0)}]
\equicnl[2]{15};
\draw[-{Stealth[length=2mm, width=1mm]}, line width = 0.7pt] (N8) -- (N1);
\draw (N1) node[right] {$r$};
\draw (N8) node[left] {$s$};
\draw[-{Stealth[length=2mm, width=1mm]}, line width = 0.7pt] (N1) -- (N4);
\draw[-{Stealth[length=2mm, width=1mm]}, line width = 0.7pt] (N7) -- (N5);
\path (N9) edge [loop right,line width = 0.7pt] (N9);
\path (N15) edge [loop left,line width = 0.7pt] (N15);
\draw (N14) node[right] {$v_i^{-1}(r)$};
\draw (N10) node[above, shift={(-0.1,0.1)}] {$v_i(s)$};
\draw[-{Stealth[length=2mm, width=1mm]}, line width = 0.7pt] (N14) -- (N8);

\draw[-{Stealth[length=2mm, width=1mm]}, blue, dashed, line width = 1.2pt] (N1) -- (N8);
\draw[-{Stealth[length=2mm, width=1mm]}, blue, dashed, line width = 1.2pt] (N8) -- (N5);
\draw[-{Stealth[length=2mm, width=1mm]}, blue, dashed, line width = 1.2pt] (N10) -- (N1);
\end{scope}
\end{tikzpicture}
        \caption{An illustration for case (VI). The black edges are part of the diagram for $w_i$ while the blue edges are part of the diagam for $w_{i+1}$. On the left, $w_i^{-1}(s) \neq v_i^{-1}(r)$ and on the right we have the situation where $w_i^{-1}(s) = v_i^{-1}(r)$.}
        \label{fig:A-case6}
    \end{figure}
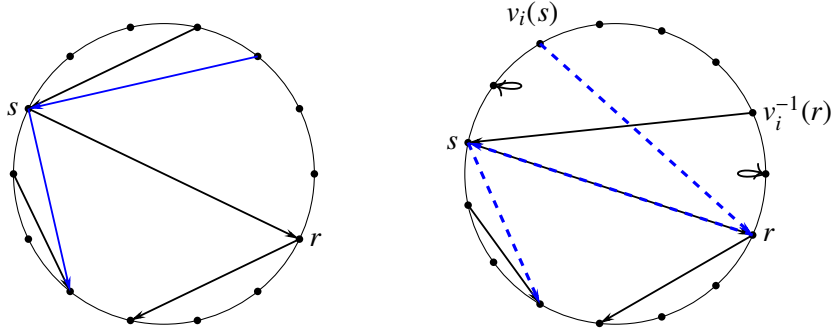
\end{enumerate}
\end{proof}

To count the number of $A$-splittable permutations, we make a psychological shift to work with \emph{ascents} instead. 

\begin{definition}
An \emph{ascent} of a permutation $w \in \mathfrak{S}_n$ is an element $i \in [n]$ such that $w(i) < w(i+1)$.
\end{definition}

The following result is folklore, but since the proof is short we include it for completeness. 
\begin{lemma}
The number of permutations $w \in \mathfrak{S}_n $ with exactly $n-2$ ascents is given by $2^n - n-1$. 
\end{lemma}
\begin{proof}
For simplicity of notation in the proof, note that by reversing the permutation it suffices to show that the number of permutations with exactly one ascent is given by $2^n - n -1$. 

Let $f(n)$ be the number of permutations with exactly one ascent. We will prove this lemma via induction, with base case $n = 1 $ vacuously true. For the inductive step, observe that it suffices to prove the recurrence $f(n) = f(n-1) + 2^{n-1} - 1$. This recurrence follows by studying where $n$ lies when considering the single-line notation of the permutation $w$. If $n$ is the first element in the single-line notation, then it follows that the remaining $n-1$ elements need to have an ascent, and there are $f(n-1)$ such permutations. Otherwise, if $n$ is not the first element then the element before $n$ forms an ascent and there cannot be any other ascents. In other words, we can split $\{1, \ldots, n-1 \} = A \sqcup B$ into two sets: $A$ contains elements occurring before $n$, and $B$ contains elements occurring after $n$. Within $A$ and $B$, the elements occur in decreasing order. Furthermore, $A \neq \emptyset$. There are $2^{n-1} - 1$ ways to make such a partition: pick which of $A$ and $B$ each of the $n-1$ elements lie in, but we need to remove one of the possibilities where $A = \emptyset$.
\end{proof}

\begin{lemma}
The \emph{Foata bijection} provides a mapping between the permutations $w \in \mathfrak{S}_n$ with exactly $n-2$ antiexceedances and the number of permutations $w' \in \mathfrak{S}_n$ with exactly $n-2$ ascents.
\end{lemma}

We omit the proof of this lemma since it is well-known. The bijection transfers antiexceedances to ascents. 

Combining the two previous lemmas, we have the following classical result. 

\begin{corollary}\label{cor:antiexc}
The number of permutations $w \in \mathfrak{S}_n$ with exactly $n-2$ antiexceedances is given by $2^n - n - 1$. 
\end{corollary}

\begin{proof}[Proof of second part of Theorem~\ref{thm:type-A}]
By Corollary~\ref{cor:antiexc}, there are $2^n - n - 1$ permutations with exactly $n-2$ antiexceedances. Of these $2^n -n - 1$ permutations $w$ we need to count the subset of them such that $\cyc(\pi_T(w)) = 1$. If $\pi_T(w)$ has more than one cycle and has exactly $n-2$ antiexceedances, then it necessarily has exactly two cycles. In particular, there exists some indices $i,j$ with $i \leq j$ such that $(i \enspace i+1 \enspace \cdots \enspace j$) form one cycle in $\pi_T(w)$ while $(j+1 \enspace \cdots \enspace n \enspace 1 \enspace \cdots \enspace i-1)$ forms the other. Note however that when $i = j$, we have that $\cyc(\pi_T(w)) = 1$ and such $w$ are actually $A$-splittable. In total, the number of permutations $w$ such that $\cyc(\pi_T(w)) = 1$ is given by $2^n - n - 1 - \binom{n}{2} + n = 2^n - \binom{n}{2} - 1$, where the $-\binom{n}{2}$ comes from choosing the indices $i,j$ and the $+n$ comes from when $i = j$.  
\end{proof}

\section{Type B}\label{sec:B}

The \emph{hyperoctahedral group} $B_n$ is the group of permutations $w$ of $\pm[n]$ such that $w(\overline{i}) = \overline{w(i)}$. In this section, we replace negative signs with bars $\overline{\bullet}$ for ease of notation. We fix the Coxeter element $c = (\bar{1} \enspace \bar{2} \enspace \cdots \enspace \bar{n} \enspace 1 \enspace 2 \enspace \cdots \enspace n)$ of $B_n$. 

The argument in this case is similar to that in type A. It is of interest whether we can prove the result for type B via the operation of folding (as described for instance in \cite[Section 4.3]{DN21}) while using the result for type A as a black-box. Here we content ourselves with modifying the argument in the previous section appropriately. 

In fact, the combinatorial model for $B_n$ can be vaguely thought of as the combinatorial model for $A_{2n}$ ``with $180^{\circ}$ rotation symmetry''. More precisely, let us briefly expound upon the combinatorial model for $NC(B_n,c)$. Label $2n$ equidistant points on a circle with the numbers $\overline{1}, \overline{2}, \ldots, \overline{n}, 1, 2, \ldots, n$ in clockwise order. 

\begin{definition}
A set partition $\pi$ of $\pm[n]$ is a \emph{type-$B_n$ noncrossing set partition} if the resulting diagram from drawing the convex hull of each block is invariant under $180^{\circ}$ rotation and the convex hulls of different blocks do not intersect each other.
\end{definition}

As in the case for type $A$ Coxeter groups, for a given type-$B_n$ noncrossing set partition, we may form an element of $NC(B_n, c)$ through the conversion of each block to a cycle by ordering the elements of the block cyclically in clockwise order around the circle. The product of these cycles is a noncrossing partition in $NC(B_n, c)$ and every noncrossing partition in $NC(B_n, c)$ arises uniquely in this way. Similarly to the case of type $A$ Coxeter groups, the partial order on $NC(B_n,c)$ corresponds to the reverse refinement order on type-$B_n$ noncrossing set partitions. The noncrossing projection of $w \in B_n$ is similarly the smallest noncrossing partition $v \in NC(B_n,c)$ such that every cycle in $w$ is contained in a cycle in $v$ (as sets).

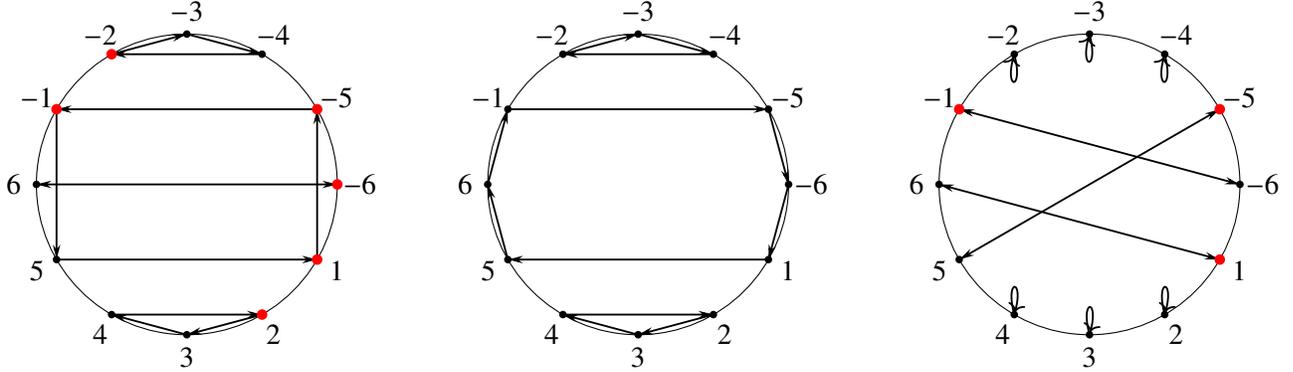
\begin{figure}
    \centering
    \begin{tikzpicture}
\begin{scope}
\equib[2cm]{6};
\draw [-{Stealth[length=2mm, width=1mm]}, line width = 0.7pt] (N2) -- (N3);
\draw [-{Stealth[length=2mm, width=1mm]}, line width = 0.7pt] (N3) -- (N4);
\draw [-{Stealth[length=2mm, width=1mm]}, line width = 0.7pt] (N4) -- (N2);

\draw [-{Stealth[length=2mm, width=1mm]}, line width = 0.7pt] (P2) -- (P3);
\draw [-{Stealth[length=2mm, width=1mm]}, line width = 0.7pt] (P3) -- (P4);
\draw [-{Stealth[length=2mm, width=1mm]}, line width = 0.7pt] (P4) -- (P2);

\draw [-{Stealth[length=2mm, width=1mm]}, line width = 0.7pt] (P1) -- (N5);
\draw [-{Stealth[length=2mm, width=1mm]}, line width = 0.7pt] (N5) -- (N1);
\draw [-{Stealth[length=2mm, width=1mm]}, line width = 0.7pt] (N1) -- (P5);
\draw [-{Stealth[length=2mm, width=1mm]}, line width = 0.7pt] (P5) -- (P1);

\draw [{Stealth[length=2mm, width=1mm]}-{Stealth[length=2mm, width=1mm]}, line width = 0.7pt] (N6) -- (P6);

\foreach \i in {1,2}{
    \fill[red](N\i) circle (0.07 cm);
}

\foreach \i in {1,2,5,6}{
    \fill[red](P\i) circle (0.07 cm);
}
\end{scope}

\begin{scope}[shift = {(6,0)}]
\equib[2cm]{6};
\draw [-{Stealth[length=2mm, width=1mm]}, line width = 0.7pt] (N2) -- (N3);
\draw [-{Stealth[length=2mm, width=1mm]}, line width = 0.7pt] (N3) -- (N4);
\draw [-{Stealth[length=2mm, width=1mm]}, line width = 0.7pt] (N4) -- (N2);

\draw [-{Stealth[length=2mm, width=1mm]}, line width = 0.7pt] (P2) -- (P3);
\draw [-{Stealth[length=2mm, width=1mm]}, line width = 0.7pt] (P3) -- (P4);
\draw [-{Stealth[length=2mm, width=1mm]}, line width = 0.7pt] (P4) -- (P2);

\draw [-{Stealth[length=2mm, width=1mm]}, line width = 0.7pt] (N1) -- (N5);
\draw [-{Stealth[length=2mm, width=1mm]}, line width = 0.7pt] (N5) -- (N6);
\draw [-{Stealth[length=2mm, width=1mm]}, line width = 0.7pt] (N6) -- (P1);
\draw [-{Stealth[length=2mm, width=1mm]}, line width = 0.7pt] (P1) -- (P5);
\draw [-{Stealth[length=2mm, width=1mm]}, line width = 0.7pt] (P5) -- (P6);
\draw [-{Stealth[length=2mm, width=1mm]}, line width = 0.7pt] (P6) -- (N1);
\end{scope}

\begin{scope}[shift = {(12,0)}]
\equib[2cm]{6};
\draw [{Stealth[length=2mm, width=1mm]}-{Stealth[length=2mm, width=1mm]}, line width = 0.7pt] (P1) -- (P6);
\draw [{Stealth[length=2mm, width=1mm]}-{Stealth[length=2mm, width=1mm]}, line width = 0.7pt] (N1) -- (N6);
\draw [{Stealth[length=2mm, width=1mm]}-{Stealth[length=2mm, width=1mm]}, line width = 0.7pt] (P5) -- (N5);
\end{scope}

\path (P2) edge [loop below,line width = 0.7pt] (P2);
\path (P3) edge [loop below,line width = 0.7pt] (P3);
\path (P4) edge [loop below,line width = 0.7pt] (P4);
\path (N2) edge [loop above,line width = 0.7pt] (N2);
\path (N3) edge [loop above,line width = 0.7pt] (N3);
\path (N4) edge [loop above,line width = 0.7pt] (N4);

\foreach \i in {1}{
    \fill[red](N\i) circle (0.07 cm);
}

\foreach \i in {1,5}{
    \fill[red](P\i) circle (0.07 cm);
}
\end{tikzpicture}
    \caption{The permutation $w:= (2 \enspace 3 \enspace 4)(\bar{2} \enspace \bar{3} \enspace \bar{4})(6 \enspace \bar{6}) (1 \enspace \bar{5} \enspace \bar{1} \enspace 5) \in B_6$ (left), its non-crossing projection $\pi_T(w) = (2 \enspace 3 \enspace 4)(\bar{2} \enspace \bar{3} \enspace \bar{4}) (1 \enspace 5 \enspace 6 \enspace \bar{1} \enspace \bar{5} \enspace \bar{6})$ (middle), and its image under $\Popt$ given by $\Popt(w) = (1 \enspace 6)(\bar{1} \enspace \bar{6})(5 \enspace \bar{5}) (2)(3)(4)(\bar{2})(\bar{3})(\bar{4})$ (right)}
    \label{fig:B-waexc}
\end{figure}

Define a total order on $\pm[n]$ by 
\[ \overline{1} \prec \overline{2} \prec \cdots \prec \overline{n} \prec 1 \prec 2 \prec \cdots \prec n. \]
In this setting, we have a similar notion of antiexceedance as in Definition~\ref{def:aexc-b}, paralleling that in $A_n$, where we replace $<$ with $\prec$. 

We have an equivalent of Lemma~\ref{lem:DN-A} in the setting of $B_n$. An illustration of this result is in Figure~\ref{fig:B-waexc}, where the antiexceedances are the vertices marked in red.

\begin{lemma}[{\cite[Theorem 5.6]{DN21}}]\label{lem:DN-B}
The following properties relating antiexceedances and $\Popt$ hold:
\begin{enumerate}
    \item[(a)] The element $c^{-1}$ has $2n-1$ antiexceedances. 
    \item[(b)] Every element of $B_n$ other than $c^{-1}$ has at most $2n-2$ antiexceedances.
    \item[(c)] For every $w \in B_n$, we have $\aexc(\Popt(w)) = \aexc(w) - \cyc(\pi_T(w))$. 
\end{enumerate}
\end{lemma}
In the setting of $B_n$, recall Definition~\ref{def:b-split} that an element $w \in B_n$ is \emph{$B$-splittable} if it contains exactly $2n-2$ antiexceedances and $\cyc(\pi_T(w)) = 1$. We recall Theorem~\ref{thm:type-B}.

\typeB*

The proof of the first part of Theorem~\ref{thm:type-B} is analogous to that for type $A$. This is because the only condition that we need to verify is that when $w \in B_n$ is $B$-splittable then $\cyc(\pi_T(\Popt^i(w))) = 1$ for all $i \geq 0$. We can run the exact same proof as that for the first part of Theorem~\ref{thm:type-A} where we can reindex and consider $w \in B_n$ as a permutation of $[2n]$ and apply Theorem~\ref{thm:type-A} on $[2n]$. However, because of the $180^{\circ}$-rotation symmetry of elements in $B_n$, we need to be slightly more careful when we count the number of $B$-splittable elements. 

Combining Theorem~\ref{thm:type-B} with \cite[Theorem 5.1]{DN21}, we prove \cite[Conjecture 5.8]{DN21}. 

\begin{corollary}[{\cite[Conjecture 5.8]{DN21}}]
The number of elements of $B_n$ that require exactly $2n-2$ or $2n-1$ iterations of $\Popt$ to reach the identity is $2^n - n$. 
\end{corollary}

\begin{proof}[Proof of second part of Theorem~\ref{thm:type-B}]

First, note that if $w \in B_n$ is $B$-splittable and has more than one cycle then it contains exactly two cycles. Next, observe that in order for $\cyc(\pi_T(w)) = 1$, if $w$ has two cycles, then for each of the cycles $\mathscr{C}$ there exists an index $i$ such that $i, \overline{i} \in \mathscr{C}$. This is simply because otherwise $\mathscr{C}_1$ would consist of $1, 2, \ldots, n $ and $\mathscr{C}_2$ would consist of $\overline{1}, \overline{2}, \ldots, \overline{n}$ and it is evident that for any such element $w$, we have that $\cyc(\pi_T(w)) = 2$. Consequently, for each cycle $\mathscr{C}$ in $w \in B_n$, there exists $\overline{i} \in \mathscr{C}$ such that $w^{-1}(\overline{i}) = j$ for some $j \in [n]$. By definition of $B_n$, we also necessarily have that $w(\overline{j}) = i$. As such, the two elements that are not antiexceedances are precisely given by $w(\overline{j}_1) = i_1$ and $w(\overline{j}_2) = i_2$, where we have one such antiexceedance in each of the two cycles. Every other element of $w$ is therefore an antiexccedance. Observe further that $\min \{ \overline{j}_1, \overline{j}_2 \} = 1$. 

The previous paragraph allows us to count $B$-splittable elements as follows. Note that every element of $w \in B_n$ can be written in a ``single-line notation'' where we only keep track of $w(i)$ for $i \in \{ \overline{n}, \overline{n-1}, \ldots, \overline{1} \}$, which in turn specifies $w(i)$ since $w(i) = \overline{w(\overline{i})}$. For instance, $1\enspace 4 \enspace \overline{3} \enspace 5 \enspace \overline{2}$ corresponds to the element $w \in B_5$ such that $w(\overline{1}) = 1$, $w(\overline{2}) = 4$, $w(\overline{3}) = \overline{3}$, $w(\overline{4}) = 5$, $w(\overline{5}) = \overline{2}$, $w(1) = \overline{1}$, $w(2) = \overline{4}$, $w(3) = 3$, $w(4) = \overline{5}$, $w(5) = 2$. We can map every element of $B_n$ into an element of $\mathfrak{S}_n$ via this single-line notation as follows: first remove the bar $\overline{\bullet}$ on any element of $w$ in the single-line notation that has a bar and then consider the resulting permutation as a single-line permutation representation of an element of $\mathfrak{S}_n$. Continuing with the example from earlier, $1\enspace 4 \enspace \overline{3} \enspace 5 \enspace \overline{2} \in B_5$ maps to $1\enspace4\enspace3\enspace5\enspace2 \in \mathfrak{S}_5$. By our observations in the previous paragraph, after this map we get a permutation of $\mathfrak{S}_n$ with at most two elements that are not antiexceedances. 

If the output of our map is $w' \in \mathfrak{S}_n$ with exactly one element that is not an antiexceedance, then it is necessarily $(n \enspace n-1 \enspace \cdots \enspace 1)$ by Lemma~\ref{lem:DN-A}. In particular, the element $w \in B_n$ we started off with would necessarily be of the form $n \enspace \overline{1} \enspace \overline{2} \enspace \cdots \enspace \overline{i-1} \enspace i \enspace \overline{i+1} \enspace \cdots \enspace \overline{n-1}$ in single-line notation by the observations we made in the first paragraph of the proof. But then $\cyc(\pi_T(w)) = 2$ because $w$ is made up of two cycles given by $(\overline{1} \enspace \overline{2} \enspace \cdots \enspace \overline{i} \enspace i+1 \enspace i+2 \enspace \cdots \enspace n)$ and $(1 \enspace 2 \enspace \cdots \enspace i \enspace \overline{i+1} \enspace \cdots \enspace \overline{n})$. In other words, if we started with $w \in B_n$ being $B$-splittable then this case is not possible. 

In the other case, we obtain $w' \in \mathfrak{S}_n$ with exactly two elements that are not antiexceedances, and we claim that we can uniquely recover the element $w \in B_n$ that we started with and that $\cyc(\pi_T(w)) = 1$. Consequently, the number of $B$-splittable elements in $B_n$ is exactly given by $2^n - n - 1$. To prove the claim, suppose $w'$ is such that $w'(i) = x_i$ and that $x_j > j$ for $j \neq 1$. The observation in the first paragraph of the proof implies that $w(\overline{1}) = x_1$ and $w(\overline{i}) = \overline{x_i}$ for $i \neq j$ while $w(\overline{j}) = x_j$. Note that if $w$ is made up of two cycles, then as sets the two cycles are $\{ x_1, x_2, \ldots, x_{\ell}, \overline{x_1}, \overline{x_2}, \ldots, \overline{x_{\ell}}\}$ and $\{ y_1, y_2, \ldots, y_{n - \ell}, \overline{y_1}, \overline{y_2}, \ldots, \overline{y_{n-\ell}}\}$ which then easily implies the desired conclusion of $\cyc(\pi_T(w)) = 1$. 
\end{proof}

\section{Type D}\label{sec:D}

The type $D_n$ case is rather subtle, because the combinatorial model we work with of the non-crossing partition lattice (based on \cite{AR04}) is more complex. We begin with a quick recap on this combinatorial model of $D_n$ as permutations on $\pm[n]$. The Coxeter group $D_n$ is the group of permutations $w \colon \pm[n] \to \pm[n]$ such that $w(\overline{i}) = \overline{w(i)}$ (where $\overline{\bullet}$ replaces the negative sign, for simplicity of notation) for all $i \in \pm[n]$ and $\left|\{ i \in [n]: w(i) < 0\} \right|$ is even. The corresponding set $T$ of reflections in $D_n$ is generated by elements of the form $(i \enspace j)(\overline{i} \enspace \overline{j})$ for some $i,j \in \pm[n]$ with $i \neq j$. In this section, we fix the Coxeter element $c = (\overline{1} \enspace \overline{2} \enspace \cdots \enspace \overline{n-1} \enspace 1 \enspace 2 \enspace \cdots \enspace n-1)(\overline{n} \enspace n)$. For a cycle $\mathcal{C} = (c_1 \enspace c_2 \enspace \cdots \enspace c_r)$, we will denote $\overline{\mathcal{C}}$ to be the cycle obtained by negating every element of $\mathcal{C}$; that is, we define $\overline{\mathcal{C}} = (\overline{c_1} \enspace \overline{c_2} \enspace \cdots \enspace \overline{c_r})$.

\begin{definition}
A cycle $\mathcal{C}$ is \emph{balanced} if $\mathcal{C} = \overline{\mathcal{C}}$. 
\end{definition}

An easy observation is that any element of $D_n$ has an even number of balanced cycles.

In \cite{AR04}, Athanasiadis and Reiner give an explicit bijection between type $D_n$ noncrossing set partitions and $\mathrm{NC}(D_n, c)$. In the remainder of this section, we abuse notation and freely interchange between noncrossing set partitions and their corresponding permutations. 

We give a brief description of the noncrossing partitions in type $D$ al\'a Athanasiadis and Reiner. Label $2n-2$ equidistant points on a circle with the numbers $\overline{1}, \overline{2}, \ldots, \overline{n-1}, 1, 2 \ldots, n-1$ in clockwise order. Next place an additional point at the center of the circle and label it with $n$ and $\overline{n}$. 

\begin{definition}
A set partition $\pi$ of $\pm[n]$ is a \emph{type-$D_n$ noncrossing set partition} if:
\begin{enumerate}
    \item [(1)] For every block $B$ of $\pi$, $\overline{B} = \{ \overline{i}: i \in B\}$ is also a block of $\pi$.
    \item [(2)] Different blocks of $\pi$ have disjoint interiors of their respective convex hulls.
    \item [(3)] The set $\{ n, \overline{n} \}$ is not a block of $\pi$.
\end{enumerate}
\end{definition}

\begin{definition}
A \emph{zero block} of a set partition $\pi$ of $\pm[n]$ is a block $B$ of $\pi$ such that $B = \overline{B}$.
\end{definition}

Observe that if $\pi$ is a type-$D$ noncrossing set partition, then $\pi$ has at most one zero block and when it does contain a zero block then the zero block properly contains $\{n, \overline{n} \}$. 

Now, given $\pi$ which is a type-$D$ noncrossing set partition, we describe how to recover the corresponding non-crossing partition in $NC(D_n, c)$, following \cite{AR04}. For each block $B$ of $\pi$ which is not zero block, form a cycle by reading the elements of $B$ in clockwise order around the boundary of the convex hull of $B$. If $\pi$ has a zero block $Z$, then form a cycle by reading the elements of $Z\setminus \{n, \overline{n} \}$ in clockwise order around the boundary of the convex hull of $Z$, and also form the cycle $\{n, \overline{n} \}$.

The following lemma from \cite{AR04} is helpful for figuring out the cycle structure of noncrossing projections of elements. 

\begin{lemma}[{\cite{AR04}}]\label{lem:AR-proj}
Let $w \in D_n$, and consider $i,j \in \pm [n]$ with $i \not \in \{j, \overline{j} \}$. We have $(i \enspace j)( \overline{i} \enspace \overline{j}) \leq_T w$ if and only if $i$ and $j$ are in the same cycle of $w$, or $i$ and $j$ are in different balanced cycles of $w$.
\end{lemma}

A consequence of this lemma is the following.

\begin{lemma}[{\cite[Lemma 6.4]{DN21}}]\label{lem:DN-6.4}
Let $u \in D_n$ and $v = \pi_T(u)$. Consider distinct $i,j \in \pm[n-1]$. If $i$ and $j$ are in the same cycle of $u$, then they are in the same cycle of $v$. 
\end{lemma}

We will often utilize the above lemma in the following rephrased form.

\begin{lemma}[Rephrased Lemma~\ref{lem:DN-6.4}]\label{cor:DN-refinement}
Let $u \in D_n$ and $v = \pi_T(u)$. Then for every cycle $\mathcal{C} \in u$, there exists a cycle $\mathcal{C}' \in v$ such that $\mathcal{C} \cap \pm[n-1] \subset \mathcal{C}'$ where the containment here is as sets. 
\end{lemma}

We will also need the following characterization of covering relations in this non-crossing partition lattice for $D_n$. 

\begin{lemma}[{\cite[Section 3]{AR04}}]\label{lem:cover-Dn}
We have that $u$ covers $v$ in the non-crossing partition lattice for $D_n$ if and only if $v$ can be obtained from $u$ via the following operations:
\begin{itemize}
    \item [(i)] splitting the zero block of $u$ into the zero block of $v$ and also a pair of nonzero blocks,
    \item [(ii)] splitting a pair of nonzero blocks of $u$ into two such pairs for $v$, or
    \item [(iii)] splitting the zero block of $u$ into one pair of nonzero blocks for $v$. 
\end{itemize}
\end{lemma}

Recall that there are two types of non-crossing set partitions, where either the non-crossing set partition contains a zero block, or it does not. We will call the non-crossing set partitions that contain a zero block \emph{zeroed}. As before, we will define an analogue of splittable elements in this setting. To do so, we need to introduce an additional piece of notation. For a permutation $w \in D_n$, we say that $\sgn(w(i)) = 1$ if $w(i) \in [n]$ and $\sgn(w(i)) = -1$ if $w(i) \in -[n]$.

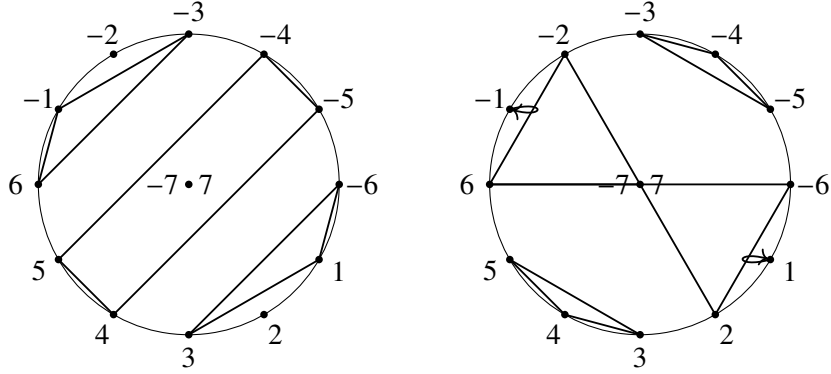
\begin{figure}[!htp]
    \centering
    \begin{tikzpicture}
\begin{scope}
\equid[2]{6}
\draw[-, line width = 0.7pt] (N1) -- (N3);
\draw[-, line width = 0.7pt] (N3) -- (P6);
\draw[-, line width = 0.7pt] (P6) -- (N1);

\draw[-, line width = 0.7pt] (P1) -- (P3);
\draw[-, line width = 0.7pt] (P3) -- (N6);
\draw[-, line width = 0.7pt] (N6) -- (P1);

\draw[-, line width = 0.7pt] (N4) -- (N5);
\draw[-, line width = 0.7pt] (N5) -- (P4);
\draw[-, line width = 0.7pt] (P4) -- (P5);
\draw[-, line width = 0.7pt] (P5) -- (N4);
\end{scope}
\begin{scope}[shift={(6,0)}]
\equid[2]{6}
\path (N1) edge [loop left,line width = 0.7pt] (N1);
\path (P1) edge [loop right,line width = 0.7pt] (P1);

\draw[-, line width = 0.7pt] (N3) -- (N4);
\draw[-, line width = 0.7pt] (N4) -- (N5);
\draw[-, line width = 0.7pt] (N5) -- (N3);
\draw[-, line width = 0.7pt] (P3) -- (P4);
\draw[-, line width = 0.7pt] (P4) -- (P5);
\draw[-, line width = 0.7pt] (P5) -- (P3);

\draw[-, line width = 0.7pt] (ori) -- (N6);
\draw[-, line width = 0.7pt] (N6) -- (P2);
\draw[-, line width = 0.7pt] (P2) -- (ori);
\draw[-, line width = 0.7pt] (ori) -- (N6);
\draw[-, line width = 0.7pt] (P6) -- (N2);
\draw[-, line width = 0.7pt] (N2) -- (ori);
\draw[-, line width = 0.7pt] (ori) -- (P6);
\end{scope}
\end{tikzpicture}
    \caption{The noncrossing set partition on the left is \emph{zeroed}, while the one on the right is not.}
    \label{fig:noncrossing-D}
\end{figure}

In this section, our goal is to enumerate and characterize elements of $D_n$ with $\Popt$ forward orbits of length $2n-2$. Before we proceed, we begin by discussing several properties of elements of $D_n$ with length $2n-2$ forward orbits. 

\begin{lemma}[{\cite[Lemma 6.5]{DN21}}]\label{lem:DN-6.5}
Suppose $j \in \pm[n-1]$ is such that $w^{-1}(j) \not \in \{ \overline{n}, n\}$. Then $(\Popt^{2n-4}(w))^{-1}(j)$ is either $j$ or the predecessor of $j$. 
\end{lemma}

\begin{corollary}\label{cor:struc-D}
Suppose $w \in D_n$ is such that $\left| O_{\Popt}(w) \right| = 2n-2$. Let $m = w(n)$. If $m \not \in \{n, \overline{n} \}$, then $\Popt^{2n-4}(w)(m) \neq m$.  
\end{corollary}

\begin{proof}
If $m \not \in \{n, \overline{n} \}$, then we may apply Lemma~\ref{lem:DN-6.5} that $\Popt^{2n-4}(w)(n) = e$, which contradicts the assumed forward orbit length of $w$ under $\Popt$. 
\end{proof}

The first step is to show that we may restrict our attention to zeroed elements.

\begin{lemma}\label{lem:not-zeroed}
Let $w \in D_n$ be such that $\pi_T(w)$ is not zeroed. Then $\left| O_{\Popt}(w) \right| \leq 2n-3$. 
\end{lemma}

In the proof of the lemma, we will need the following result from \cite{DN21}.

\begin{lemma}[{\cite[Lemma 6.3]{DN21}}]\label{lem:DN-zero}
Suppose $v \in NC(D_n,c)$ and $k \in \pm[n-1]$ is such that $v(n) = k$. Let $Y$ be the set consisting of the next $n-1$ numbers after $k$ in clockwise order around the circle. Then none of the elements of $Y$ lie in the same cycle as $k$ in $v$. 
\end{lemma}

\begin{proof}[Proof of Lemma~\ref{lem:not-zeroed}]
Suppose for the sake of contradiction that $w \in D_n$ is such that $\pi_T(w)$ is not zeroed but $\left| O_{\Popt}(w) \right| = 2n-2$. For $k \geq 0$, let $ u_k = \Popt^k(w)$ and $v_k = \pi_T(u_k)$. Let $ m = w(n)$. First, we claim that since $\pi_T(w)$ is not zeroed, $m \not \in \{n, \overline{n} \}$. If $m = n$, we have that $v_k(n) = n$ for all $k \geq 0$. This means that the action of $\Popt$ is only on $w|_{n-1}$, and we may effectively think of $w |_{n-1}$ as lying in $B_{n-1}$; given a noncrossing partition of $D_n$ that fixes $n$ and $\overline{n}$, we may uniquely map it to an induced noncrossing partition of $B_{n-1}$ by restricting the partition to $\pm[n-1]$. This means that the set of noncrossing partitions of $D_n$ that fix $n$ and $\overline{n}$ can be mapped to a subset of the noncrossing partitions of $B_{n-1}$. This in turn implies that the number of iterations of $\Popt$ needed for $w$ to reach $e$ is at most the number of iterations of $\Popt(\bullet, \widetilde{c})$ for $w|_{n-1}$ to reach $e$ where $\widetilde{c} = (\overline{1} \enspace \overline{2} \enspace \cdots \enspace \overline{n-1} \enspace 1 \enspace 2 \enspace \cdots \enspace n-1)$. Now, $w|_{n-1} \neq \widetilde{c}$ because $\widetilde{c}$ is a balanced cycle and if $w|_{n-1} = \widetilde{c}$, then $w$ would not have an even number balanced cycles, a contradiction to $w \in D_n$. But by \cite[Theorem 5.1]{DN21}, since $w|_{n-1} \neq \widetilde{c}$ and $\pi_T(w)$ is not zeroed, we have that the maximum possible length of the forward orbit of $w |_{n-1}$ is $2(n-1) -1 = 2n-3$, a contradiction.

If $m = \overline{n}$, then we also have that $w(\overline{n}) = n$ so that $(n \enspace \overline{n})$ is a balanced cycle in $w$. Since $w \in D_n$, there must exist another balanced cycle $\mathcal{C}$. Suppose $v \in \mathcal{C}$. By Lemma~\ref{lem:AR-proj}, it follows that $(v \enspace n)(\overline{v} \enspace \overline{n}) \leq_T w \leq_T \pi_T(w)$. Since $\pi_T(w)$ is not zeroed, it follows that $n$ and $\overline{n}$ are not in the same cycle and therefore $n$ and $v$ must be in the same cycle of $\pi_T(w)$ by applying Lemma~\ref{lem:AR-proj}. Since $\mathcal{C}$ is balanced, $\overline{v} \in \mathcal{C}$ as well. By considering $(v \enspace \overline{n})(\overline{v} \enspace n) \leq_T \pi_T(w)$, the same argument gives that $n$ and $\overline{v}$ must lie in the same cycle in $\pi_T(w)$ as well. However, this would then be a contradiction to Lemma~\ref{lem:DN-zero}. As such, we have shown that $m \not \in \{n, \overline{n} \}$.

Observe that if $u_t(j) = j$ for some $t$ then for all $s \geq t$ we have that $u_s(j) = j$ as well, since $(j)$ will be in a singleton cycle in $v_s$ for all $s \geq t$. Consequently, by Corollary~\ref{cor:struc-D}, to get a contradiction it suffices to prove that there exists some $s \leq 2n-4$ such that $u_s(w)(m) = m$. In fact, we will prove something stronger; we will show that there exists some $s \leq n-1$ such that $u_s(w)(m) = m$. 

To that end, we introduce some notation from \cite{DN21}. Define a total order $\prec_j$ that is defined by reading the numbers clockwise around the circle so that the last number read is $j$. Let $\mathcal{C}_i$ be the cycle in $v_i$ containing $j$. 

Since $\pi_T(w)$ is not zeroed, by Lemma~\ref{lem:cover-Dn} it follows that $v_i$ is not zeroed for all $i \geq 1$. Since there are no balanced cycles in $v_i$, by Lemma~\ref{lem:AR-proj} if $\alpha$ and $\beta$ are in the same cycle of $u_i$, then they must lie in the same cycle of $v_i$. Since $u_{i+1}^{-1} = v_iu_i^{-1}$ and $u_k^{-1}(m)$ and $m$ are in the same cycle of $u_k$, it follows that $u_{k+1}^{-1}(m)$ and $m$ are in the same cycle of $v_{k+1}$ and in particular that $u_{k+1}^{-1}(m)$ is the element after $u_k^{-1}(m)$ in the clockwise order on $\mathcal{C}$. Putting this together for $k \geq 1$ where we observe that since $m \in \pm[n-1]$, we get that
\begin{equation}\label{eq:2}
    u_1^{-1}(m) \prec_m u_2^{-1}(m) \prec_m \cdots \prec_m m
\end{equation}

By Lemma~\ref{lem:DN-zero}, we have that $\overline{m} \prec_m u_1^{-1}(m)$. Consequently, the sequence in (\ref{eq:2}) is of length at most $n-1$, and so there exists some $s \leq n-1$ such that $u_s(w)(m) = m$, as desired.
\end{proof}

Next, for a zeroed element $w \in D_n$, we associate with it a permutation in $B_{n-1}$. We will fix the Coxeter element $c = (\bar{1} \enspace \bar{2} \enspace \cdots \enspace \bar{n-1} \enspace 1 \enspace 2 \enspace \cdots \enspace n-1)$ in $B_{n-1}$. Let $a = w(n)$ and $b = w^{-1}(n)$. We will define $w|_{n-1} \in B_{n-1}$ by suitably defining $w^{-1}(a), w^{-1}(\bar{a})$ and $w(b), w(\bar{b})$ so that $w|_{n-1}$ has the same convex hull of its cycles as that of $w$. Let $w_1 \in D_n$ be defined by
\[ w_1(x) = \begin{cases} w(x) &\text{if }x \not \in \{b, \bar{b} \}, \\ a &\text{if } x = b, \\ \bar{a} &\text{if } x = \bar{b}. \end{cases} \]
Let $w_2 \in D_n$ be defined by 
\[ w_2(x) = \begin{cases} w(x) &\text{if }x \not \in \{b, \bar{b} \}, \\ \bar{a} &\text{if } x = b, \\ a &\text{if } x = \bar{b}. \end{cases} \]

Suppose the elements of the cycles $\mathcal{C}_1, \ldots, \mathcal{C}_k \in w$ make up the zero block of $\pi_T(w)$. Observe that if $ \mathcal{C}_i \cap \{n, \bar{n} \} = \emptyset$, then $\mathcal{C}_i \in w_1, w_2$ and if $ \mathcal{C}_i \cap \{n, \bar{n} \} \neq \emptyset$ then there exists some $i \in \{1,2 \}$ and a cycle $\mathcal{C} \in w_i$ such that $\mathcal{C}_i \subset \mathcal{C}$. Suppose some cycle $ \mathcal{C}_i \cap \{n, \bar{n} \} \neq \emptyset$ and the interior of the convex hulls of $\mathcal{C}_i$ and $\mathcal{C}_j$ is nonempty. Then for the corresponding $w_i$, the interior of the convex hulls of $\mathcal{C}$ and $\mathcal{C}_j$ is nonempty. It consequently follows that $\pi_T(w_i)$ is zeroed and its zero block contains the same elements as that of $\pi_T(w)$. It is clear that the nonzero blocks of $\pi_T(w_i)$ and $\pi_T(w)$ necessarily coincide. The observation that one of $w_1$ or $w_2$ has the same convex hull as $w$ allows us to make the following definition.

\begin{definition}
Given a zeroed permutation $w \in D_n$, define $w_1$ and $w_2$ as above. Define the \emph{$n-1$-projection} of $w$, written as $w|_{n-1} \in B_{n-1}$, to be either of $w_1$ or $w_2$ satisfying $\pi_T^B(w_i) = \pi_T^D(w)$. If both $w_1$ and $w_2$ have the desired property, choose the $w_i$ such that $w_i^{-1}(w(n)) = w^{-1}(n)$. Here the superscripts of the noncrossing projection indicates whether we are working with the set-up of $B_{n-1}$ or $D_n$.
\end{definition}

\begin{figure}[!htp]
    \centering
    \begin{tikzpicture}
\begin{scope}
\equidd[2]{6}
\draw[{Stealth[length=2mm, width=1mm]}-{Stealth[length=2mm, width=1mm]}, line width = 0.7pt] (N1) -- (P6);
\draw[{Stealth[length=2mm, width=1mm]}-{Stealth[length=2mm, width=1mm]}, line width = 0.7pt] (P1) -- (N6);
\draw[-{Stealth[length=2mm, width=1mm]}, line width = 0.7pt] (N2) -- (N3);
\draw[-{Stealth[length=2mm, width=1mm]}, line width = 0.7pt] (N3) -- (N4);
\draw[-{Stealth[length=2mm, width=1mm]}, line width = 0.7pt] (N4) -- (N5);
\draw[-{Stealth[length=2mm, width=1mm]}, line width = 0.7pt] (ori) -- (N2);
\draw[-{Stealth[length=2mm, width=1mm]}, line width = 0.7pt] (N5) -- (ori);

\draw[-{Stealth[length=2mm, width=1mm]}, line width = 0.7pt] (P2) -- (P3);
\draw[-{Stealth[length=2mm, width=1mm]}, line width = 0.7pt] (P3) -- (P4);
\draw[-{Stealth[length=2mm, width=1mm]}, line width = 0.7pt] (P4) -- (P5);
\draw[-{Stealth[length=2mm, width=1mm]}, line width = 0.7pt] (ori) -- (P2);
\draw[-{Stealth[length=2mm, width=1mm]}, line width = 0.7pt] (P5) -- (ori);
\end{scope}

\begin{scope}[shift={(6,0)}]
\equib[2]{6}
\draw[{Stealth[length=2mm, width=1mm]}-{Stealth[length=2mm, width=1mm]}, line width = 0.7pt] (N1) -- (P6);
\draw[{Stealth[length=2mm, width=1mm]}-{Stealth[length=2mm, width=1mm]}, line width = 0.7pt] (P1) -- (N6);
\draw[-{Stealth[length=2mm, width=1mm]}, line width = 0.7pt] (N2) -- (N3);
\draw[-{Stealth[length=2mm, width=1mm]}, line width = 0.7pt] (N3) -- (N4);
\draw[-{Stealth[length=2mm, width=1mm]}, line width = 0.7pt] (N4) -- (N5);
\draw[-{Stealth[length=2mm, width=1mm]}, line width = 0.7pt] (P5) -- (N2);
\draw[-{Stealth[length=2mm, width=1mm]}, line width = 0.7pt] (N5) -- (P2);
\draw[-{Stealth[length=2mm, width=1mm]}, line width = 0.7pt] (P2) -- (P3);
\draw[-{Stealth[length=2mm, width=1mm]}, line width = 0.7pt] (P3) -- (P4);
\draw[-{Stealth[length=2mm, width=1mm]}, line width = 0.7pt] (P4) -- (P5);

\end{scope}
\end{tikzpicture}
    \caption{On the left is a diagram representing either the permutation $w = (1 \enspace \bar{6}) (2 \enspace 3 \enspace 4 \enspace 5 \enspace 7) (\bar{1} \enspace 6)  (\bar{2} \enspace \bar{3} \enspace \bar{4} \enspace \bar{5} \enspace \bar{7}) \in D_7$ or $w' = (1 \enspace \bar{6}) (2 \enspace 3 \enspace 4 \enspace 5 \enspace \bar{7}) (\bar{1} \enspace 6)  (\bar{2} \enspace \bar{3} \enspace \bar{4} \enspace \bar{5} \enspace 7)$. On the right, we have $w|_{6} = w'|_{6} = (1 \enspace \overline{6})(\overline{1} \enspace 6)(2 \enspace 3 \enspace 4 \enspace 5 \enspace \overline{2} \enspace \overline{3} \enspace \overline{4} \enspace \overline{5})$ which should be thought of as an element of $B_6$.}
    \label{fig:proj}
\end{figure}
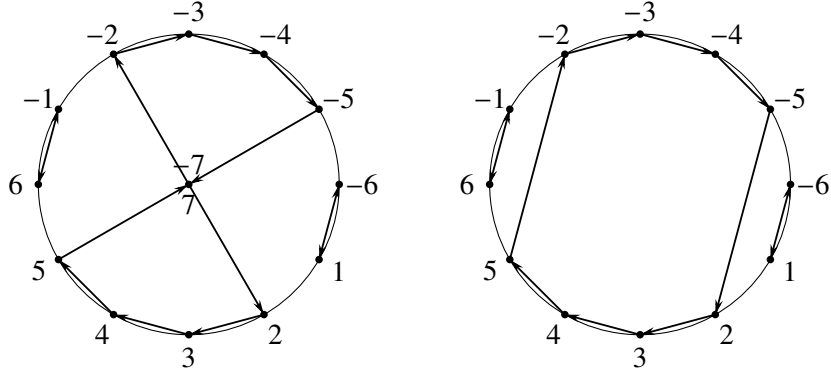

With this definition in place, we can define pre-$D$-splittable elements, that parallel the definition of splitable elements in the context of type A and type B Coxeter groups.

\begin{definition}
An element $w \in D_n$ is \emph{pre-$D$-splittable} if $\pi_T(w) = c$, and either $w = c^{-1}$ or $w|_{n-1}$ contains exactly $2n-4$ antiexceedances.
\end{definition}

It turns out in the setting of type D Coxeter groups, however, that an additional condition is needed in order for $w$ to have maximum $\Popt$ orbit of length $2n-2$. Define a total order $\prec$ on $\pm[n-1]$ by $\overline{1} \prec \cdots \prec \overline{n-1} \prec 1 \prec \cdots \prec n-1$. Define $\prec_j$ to be the order obtained by doing a cyclic shift on $\prec$ such that the maximum element of $\prec_j$ is $j$. 

\begin{definition}
An element $w \in D_n$ is \emph{$D$-splittable} if it is pre-$D$-splittable and there exists $x$ such that $x \prec_{w(n)} w^{-1}(n) $.
\end{definition}

With this notation in place, we can answer a question posed in \cite[Section 6]{DN21}, about the number of elements $w \in D_n$ with $|O_{\Popt}(w) | = 2n-2$. Note that the Coxeter number of $D_n$ is $2n-2$, so that we are enumerating elements with maximum forward orbit length. This is in contrast with the setting of type A and type B Coxeter groups, where \cite[Theorem 5.1]{DN21} shows that there exists a unique element with maximum forward orbit length. Of note is that the value put forth in \cite[Conjecture 6.6]{DN21} is incorrect; for instance there are 7 elements in $D_4$ with forward orbit length 6 as opposed to the conjectured value of 1 element as predicted by \cite[Conjecture 6.6]{DN21}. 

We recall Theorem~\ref{thm:type-D}.

\typeD*

Next, we show that if $w \in D_n$ is such that $\pi_T(w)$ is zeroed but $\pi_T(w) \neq c$ or $w|_{n-1}$ does not contain $2n-4$ antiexceedances then $w$ cannot satisfy $\left| O_{\Popt}(w) \right| = 2n-2$. A reduction that we make in this proof is to relate $\Popt(\bullet, c)$ on $D_n$ to $\Popt(\bullet, \widetilde{c})$ on $B_n$ where $\widetilde{c} = (\overline{1} \enspace \overline{2} \enspace \cdots \enspace \overline{n-1} \enspace 1 \enspace 2 \enspace \cdots \enspace n-1)$ so as to utilize our results from the previous section. To that end, we first establish an auxiliary result roughly relating the cycle structure of $\pi_T(\Popt^{D}(w),c)$ and $\pi_T(\Popt^{B}(w|_{n-1}), \widetilde{c})$ for elements $w \in D_n$ such that $\pi_T(w)$ is zeroed, where the superscripts are indicative of the setting in which we should think of $\Popt$ as acting. To more succinctly state the result, we introduce one more piece of notation.

\begin{definition}
Suppose $w |_{n-1}$ is obtained from $w \in D_n$. Let $t = w(n)$. The \emph{$n$-extension} of $\Popt^B(w |_{n-1})$ denoted $\prescript{n}{}\Popt(w |_{n-1})$ is defined to be the permutation where $\prescript{n}{}\Popt(w |_{n-1})^{-1}(j) = \Popt^B(w |_{n-1})^{-1}(j)$ for all $j \in \pm [n]$ such that $j \not \in \{ t, \overline{t}, n, \overline{n}\}$, while $\prescript{n}{}\Popt(w |_{n-1})^{-1}(t) = \overline{n}$, $\prescript{n}{}\Popt(w |_{n-1})^{-1}(\overline{t}) = n$, $\prescript{n}{}\Popt(w |_{n-1})^{-1}(n) = \Popt^B(w |_{n-1})^{-1}(\overline{t})$ and $\prescript{n}{}\Popt(w |_{n-1})^{-1}(\overline{n}) = \Popt^B(w |_{n-1})^{-1}(t)$.
\end{definition}

\begin{lemma}\label{lem:DtoB}
For any $w \in D_n$ such that $\pi_T(w)$ is zeroed and $w(n) \in \pm[n-1]$, $\Popt(w) = \prescript{n}{}\Popt(w|_{n-1})$. 
\end{lemma}

\begin{proof}
Suppose $\pi_T(B) = \mathscr{C}_1 \cdots \mathscr{C}_{\ell} (n \enspace \overline{n})$. First, we prove that $\pi_T^B(w|_{n-1}) = \mathscr{C}_1 \cdots \mathscr{C}_{\ell}$. Suppose otherwise for the sake of contradiction so that $\pi_T(w|_{n-1}) = \mathcal{C}_1 \cdots \mathcal{C}_k$ where $k \geq 2$, such that there exists at least an index $r$ with $\mathscr{C}_r \neq \mathcal{C}_r$. By Corollary~\ref{cor:DN-refinement}, we necessarily have that for every $i \in [k]$, there exists a corresponding index $q_i$ such that $\mathcal{C}_i \subset \mathscr{C}_{q_i}$ where the inclusion is as sets. Let \[ w = (c_{11} \enspace c_{12} \cdots c_{1\ell_1})  \cdots (c_{j1} \enspace c_{j2} \cdots c_{jt_1} \enspace n \enspace c_{js_1}  \cdots c_{j\ell_j}) \cdots (c_{k1} \enspace c_{k2}  \cdots c_{jt_2} \enspace \overline{n} \enspace c_{js_2} \cdots  c_{k \ell_k})  \cdots  (c_{m1} \enspace c_{m2}  \cdots  c_{m \ell_m}).\] Because by definition $w|_{n-1}$ contains a balanced cycle \[(c_{j1} \enspace \cdots \enspace c_{jt_1} \enspace c_{js_2} \enspace \cdots \enspace c_{k\ell_k}c_{k1} \enspace \cdots \enspace c_{jt_2}c_{js_1} \enspace \cdots \enspace c_{j\ell_j} ) \enspace \cdots \enspace (c_{k1} \enspace c_{k2} \enspace \cdots \enspace c_{jt_2} \enspace c_{js_2} \enspace c_{k \ell_k}),\] it follows that $\pi_T^B(w|_{n-1})$ contains a balanced cycle as well which corresponds to the part of the partition containing the convex hull of this cycle. without loss of generality this balanced cycle is $\mathcal{C}_1$. Then we have that $\mathcal{C}_2, \ldots, \mathcal{C}_k$ are disjoint from $\mathcal{C}_1$. 

Now, any cycle $\mathcal{C}$ of $w$ not containing $n, \overline{n}$ is also a cycle in $w|_{n-1}$. Next, we observe that if the convex hull of $\mathcal{C}$ in $w$ intersects the convex hull of $(c_{j1} \enspace c_{j2} \enspace \cdots \enspace c_{jt_1} \enspace n \enspace c_{js_1} \enspace \cdots \enspace c_{j\ell_j})(c_{k1} \enspace c_{k2} \enspace \cdots \enspace c_{jt_2} \enspace \overline{n} \enspace c_{js_2} \enspace \cdots \enspace c_{k \ell_k})$, then the convex hull of $\mathcal{C}$ in $w|_{n-1}$ would also intersect the convex hull of \[(c_{j1} \enspace \cdots \enspace c_{jt_1} \enspace c_{js_2} \enspace \cdots \enspace c_{k\ell_k}c_{k1} \enspace \cdots \enspace c_{jt_2}c_{js_1} \enspace \cdots c_{j\ell_j} ) \enspace \cdots \enspace (c_{k1} \enspace c_{k2} \enspace \cdots \enspace c_{jt_2} \enspace c_{js_2} \enspace c_{k \ell_k}).\]

Recall that $D_n$ reflections are given by $(i \enspace j)(\overline{i} \enspace \overline{j})$. In particular, we claim that if $(i \enspace j)(\overline{i} \enspace \overline{j}) \leq_T w $ for some $i,j \in \pm [n-1]$ and $i \not \in \{j, \overline{j} \}$, then $(i \enspace j)(\overline{i} \enspace \overline{j}) \leq_T \mathcal{C}_1 \mathcal{C}_2 \cdots \mathcal{C}_k (n \enspace \overline{n})$. To see this, first make the observations that every cycle in $w|_{n-1}$ is contained (as a set) in a cycle $\mathcal{C}_i$ and also that if $\mathcal{C}$ is a balanced cycle in $w|_{n-1}$ such that $\mathcal{C} \subset \mathcal{C}_j$ (here containment is as sets) then $\mathcal{C}_j$ is a balanced cycle as well, where $\mathcal{C}_j$ is the unique balanced cycle in $\pi_T^B(w|_{n-1})$. For $(i \enspace j)(\overline{i} \enspace \overline{j}) \leq_T w $ by Lemma~\ref{lem:AR-proj}, there are two possibilities: if $i$ and $j$ lie in the same cycle, then they are lie in the same cycle for some $\mathcal{C}_j$ by the earlier observation and so $(i \enspace j)(\overline{i} \enspace \overline{j}) \leq_T \mathcal{C}_1 \cdots \mathcal{C}_k (n \enspace \overline{n})$; if $i$ and $j$ lie in disjoint balanced cycles, then by the earlier observation they lie in the unique balanced cycle $\mathcal{C}_b$ of $\pi_T^B(w \bigr|_{n-1})$ so that $(i \enspace j)(\overline{i} \enspace \overline{j}) \leq_T \mathcal{C}_1 \cdots \mathcal{C}_k (n \enspace \overline{n})$ as well. 

At this stage, the only $D_n$ reflections $(i \enspace j)(\overline{i} \enspace \overline{j}) \leq_T w$ that we have yet to account for are those where $i \in \{n, \overline{n}\}$. There are two cases:

\begin{itemize}
    \item $n$ and $\overline{n}$ lie in the same cycle. Then the possibilities for $\{i, j \}$ are of the form $\{n, t \}$ or $\{ \overline{n}, t \}$ where $t$ is an element from a balanced cycle. By the earlier observation, $t$ continues to lie in a balanced cycle $\mathcal{C}_b$, so that $(n \enspace t)(\overline{n} \enspace \overline{t}) \leq_T \mathcal{C}_1 \mathcal{C}_2 \cdots \mathcal{C}_k (n \enspace \overline{n})$.
    \item $n$ and $\overline{n}$ lie in disjoint cycles. Then the only possibilities for $\{i, j \}$ are of the form $\{n, c_{jx} \}$ and $\{ \overline{n}, c_{ky} \}$ for appropriate indices $x,y$. It is clear that $(n \enspace c_{jx})(\overline{n} \enspace \overline{c_{jx}}) \leq_T \mathcal{C}_1 \cdots \mathcal{C}_k (n \enspace \overline{n})$. 
\end{itemize}

Lastly, $\mathcal{C}_1 \mathcal{C}_2 \cdots \mathcal{C}_k (n \enspace \overline{n})$ is a non-crossing partition of $D_n$, since $\mathcal{C}_1 \mathcal{C}_2 \cdots \mathcal{C}_k$ is a non-crossing partition of $B_n$. This means that $\pi_T(w) \leq_T \mathcal{C}_1 \mathcal{C}_2 \cdots \mathcal{C}_k (n \enspace \overline{n})$, which is a contradiction by Lemma~\ref{lem:cover-Dn}.

What remains is a computation. Let $t = w(n)$. Because $\pi_T^B(w|_{n-1}) = \mathscr{C}_1 \cdots \mathscr{C}_{\ell}$ and $\pi_T(w) = \mathscr{C}_1 \cdots \mathscr{C}_{\ell}(n \enspace \overline{n})$, it follows that $\Popt(w)^{-1}(j) = \prescript{n}{}\Popt(w|_{n-1})^{-1}(j)$ for $j \not \in \{n, \overline{n}, t, \overline{t} \}$. Now, $\Popt(w)^{-1}(t) = (\pi_T(w))^{-1}(n) = \overline{n} = \prescript{n}{}\Popt(w|_{n-1})^{-1}(n)$ and similarly $\Popt(w)^{-1}(\overline{t}) = n = \prescript{n}{}\Popt(w|_{n-1})^{-1}(n)$. Furthermore, without loss of generality $w^{-1}(n) \in \mathscr{C}_1$. Then $\Popt(w)^{-1}(n) = \mathscr{C}_1(w^{-1}(n)) = \mathscr{C}_1(w|_{n-1}^{-1}(t)) = \prescript{n}{}\Popt(w|_{n-1})^{-1}(t)$ and analogously $\Popt(w)^{-1}(\overline{n}) = \prescript{n}{}\Popt(w|_{n-1})^{-1}(\overline{t})$.
\end{proof}

The first step of the proof above shows that if $\pi_T(B) = \mathscr{C}_1 \cdots \mathscr{C}_{\ell} (n \enspace \overline{n})$, then $\pi_T^{B}(w|_{n-1}) = \mathscr{C}_1 \cdots \mathscr{C}_{\ell}$. It may be tempting to claim that the converse is also true. In this dream scenario, analyzing the behavior of $w$ under the operation of $\Popt$ would very neatly reduce to that of understanding the operation of $w|_{n-1}$ under $\Popt^B$, and we already know how to do the latter because of Section~\ref{sec:B}. However, such a claim is unfortunately false and even $\pi_T^B(w|_{n-1}) = \widetilde{c}$ is insufficient to guarantee that $\pi_T(B)$ is zeroed. For instance, consider the permutation $w = 63512 \in D_6$ (where we are writing the permutation on $\pm [n]$ in one-line notation). Then $w|_{5} = \overline{1}35\overline{4}2\in B_5 $ is such that $\pi_T^B(w|_{5}) = 2345\overline{1} = \widetilde{c}_5$ while $\pi_T(w) = 234561 \neq c_6$. This shows that the behavior of $\Popt^D$ is more subtle that we might na\"ively expect. 

\begin{figure}[!htp]
    \centering
    \begin{tikzpicture}

\begin{scope}
\equid[1.7]{8}

\draw[-{Stealth[length=2mm, width=1mm]}, line width = 0.7pt] (N4) -- (N5);
\draw[-{Stealth[length=2mm, width=1mm]}, line width = 0.7pt] (N5) -- (ori);
\draw[-{Stealth[length=2mm, width=1mm]}, line width = 0.7pt] (P5) -- (ori);
\draw[-{Stealth[length=2mm, width=1mm]}, line width = 0.7pt] (ori) -- (P4);
\draw[-{Stealth[length=2mm, width=1mm]}, line width = 0.7pt] (ori) -- (N4);
\draw[-{Stealth[length=2mm, width=1mm]}, line width = 0.7pt] (P4) -- (P5);

\draw[-{Stealth[length=2mm, width=1mm]}, line width = 0.7pt] (N1) -- (N2);
\draw[-{Stealth[length=2mm, width=1mm]}, line width = 0.7pt] (N2) -- (N8);
\draw[-{Stealth[length=2mm, width=1mm]}, line width = 0.7pt] (N8) -- (N1);

\draw[-{Stealth[length=2mm, width=1mm]}, line width = 0.7pt] (P1) -- (P2);
\draw[-{Stealth[length=2mm, width=1mm]}, line width = 0.7pt] (P2) -- (P8);
\draw[-{Stealth[length=2mm, width=1mm]}, line width = 0.7pt] (P8) -- (P1);

\draw[-{Stealth[length=2mm, width=1mm]}, line width = 0.7pt] (N3) -- (N6);
\draw[-{Stealth[length=2mm, width=1mm]}, line width = 0.7pt] (N6) -- (N7);
\draw[-{Stealth[length=2mm, width=1mm]}, line width = 0.7pt] (N7) -- (N3);

\draw[-{Stealth[length=2mm, width=1mm]}, line width = 0.7pt] (P3) -- (P6);
\draw[-{Stealth[length=2mm, width=1mm]}, line width = 0.7pt] (P6) -- (P7);
\draw[-{Stealth[length=2mm, width=1mm]}, line width = 0.7pt] (P7) -- (P3);
\end{scope}

\begin{scope}[shift={(5,0)}]
\equib[1.7]{8}
\draw[-{Stealth[length=2mm, width=1mm]}, line width = 0.7pt] (N4) -- (N5);
\draw[-{Stealth[length=2mm, width=1mm]}, line width = 0.7pt] (N5) -- (P4);
\draw[-{Stealth[length=2mm, width=1mm]}, line width = 0.7pt] (P4) -- (P5);
\draw[-{Stealth[length=2mm, width=1mm]}, line width = 0.7pt] (P5) -- (N4);

\draw[-{Stealth[length=2mm, width=1mm]}, line width = 0.7pt] (N1) -- (N2);
\draw[-{Stealth[length=2mm, width=1mm]}, line width = 0.7pt] (N2) -- (N8);
\draw[-{Stealth[length=2mm, width=1mm]}, line width = 0.7pt] (N8) -- (N1);

\draw[-{Stealth[length=2mm, width=1mm]}, line width = 0.7pt] (P1) -- (P2);
\draw[-{Stealth[length=2mm, width=1mm]}, line width = 0.7pt] (P2) -- (P8);
\draw[-{Stealth[length=2mm, width=1mm]}, line width = 0.7pt] (P8) -- (P1);

\draw[-{Stealth[length=2mm, width=1mm]}, line width = 0.7pt] (N3) -- (N6);
\draw[-{Stealth[length=2mm, width=1mm]}, line width = 0.7pt] (N6) -- (N7);
\draw[-{Stealth[length=2mm, width=1mm]}, line width = 0.7pt] (N7) -- (N3);

\draw[-{Stealth[length=2mm, width=1mm]}, line width = 0.7pt] (P3) -- (P6);
\draw[-{Stealth[length=2mm, width=1mm]}, line width = 0.7pt] (P6) -- (P7);
\draw[-{Stealth[length=2mm, width=1mm]}, line width = 0.7pt] (P7) -- (P3);
\end{scope}

\begin{scope}[shift={(10,0)}]
\equid[1.7]{8}
\draw[-{Stealth[length=2mm, width=1mm]}, line width = 0.7pt] (ori) -- (P1);
\draw[-{Stealth[length=2mm, width=1mm]}, line width = 0.7pt] (P1) -- (P2);
\draw[-{Stealth[length=2mm, width=1mm]}, line width = 0.7pt] (P2) -- (P3);
\draw[-{Stealth[length=2mm, width=1mm]}, line width = 0.7pt] (P3) -- (P4);
\draw[-{Stealth[length=2mm, width=1mm]}, line width = 0.7pt] (P4) -- (P5);
\draw[-{Stealth[length=2mm, width=1mm]}, line width = 0.7pt] (P5) -- (P6);
\draw[-{Stealth[length=2mm, width=1mm]}, line width = 0.7pt] (P6) -- (P7);
\draw[-{Stealth[length=2mm, width=1mm]}, line width = 0.7pt] (P7) -- (P8);
\draw[-{Stealth[length=2mm, width=1mm]}, line width = 0.7pt] (P8) -- (ori);

\draw[-{Stealth[length=2mm, width=1mm]}, line width = 0.7pt] (ori) -- (N1);
\draw[-{Stealth[length=2mm, width=1mm]}, line width = 0.7pt] (N1) -- (N2);
\draw[-{Stealth[length=2mm, width=1mm]}, line width = 0.7pt] (N2) -- (N3);
\draw[-{Stealth[length=2mm, width=1mm]}, line width = 0.7pt] (N3) -- (N4);
\draw[-{Stealth[length=2mm, width=1mm]}, line width = 0.7pt] (N4) -- (N5);
\draw[-{Stealth[length=2mm, width=1mm]}, line width = 0.7pt] (N5) -- (N6);
\draw[-{Stealth[length=2mm, width=1mm]}, line width = 0.7pt] (N6) -- (N7);
\draw[-{Stealth[length=2mm, width=1mm]}, line width = 0.7pt] (N7) -- (N8);
\draw[-{Stealth[length=2mm, width=1mm]}, line width = 0.7pt] (N8) -- (ori);

\end{scope}
\end{tikzpicture}
    \caption{An illustration of a \emph{double-cysted} permutation. Here, on the left we have \[w = (1 \enspace 2 \enspace 8)(3\enspace6\enspace7)(4\enspace 9 \enspace 5) (\overline{4} \enspace \overline{5} \enspace \overline{9})(\overline{1} \enspace \overline{2} \enspace \overline{8})(\overline{3} \enspace \overline{6} \enspace \overline{7}) \]  and in the middle we have \[w|_{8} = (1 \enspace 2 \enspace 8)(3\enspace6\enspace7)(4\enspace 5\enspace \overline{4} \enspace \overline{5})(\overline{1} \enspace \overline{2} \enspace \overline{8})(\overline{3} \enspace \overline{6} \enspace \overline{7}).\]  On the right is $\pi_T(w) = (1 \enspace 2 \enspace 3 \enspace 4 \enspace 5 \enspace 6 \enspace 7 \enspace 8 \enspace 9) (\overline{1} \enspace \overline{2} \enspace \overline{3} \enspace \overline{4} \enspace \overline{5} \enspace \overline{6} \enspace \overline{7} \enspace \overline{8} \enspace \overline{9})$ which consists of two cycles.}
    \label{fig:double-cyst}
\end{figure}
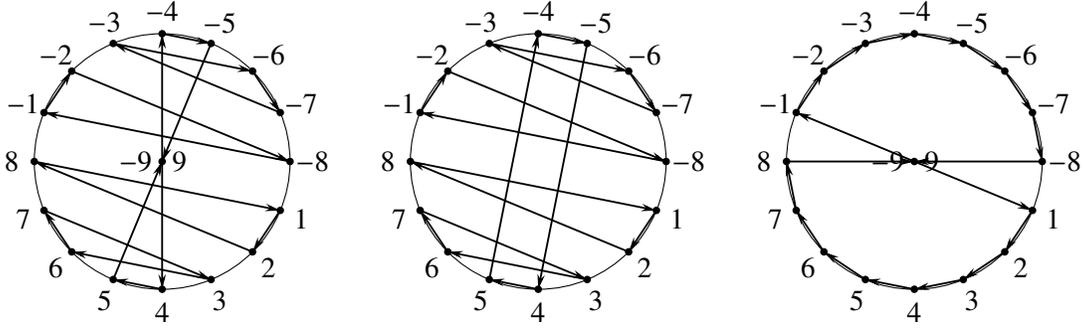

\begin{definition} 
Suppose $w |_{n-1}$ is obtained from $w \in D_n$ and let $t = w(n)$. The \emph{modified $n$-extension} of $\Popt^B(w |_{n-1})$ denoted $\prescript{n}{m}\Popt(w |_{n-1})$ is defined to be the permutation satisfying
\begin{itemize}
    \item $\prescript{n}{m}\Popt(w |_{n-1})^{-1}(j) = \Popt^B(w |_{n-1})^{-1}(j)$ for all $j \in \pm [n]$ such that $j \not \in \{ t, \overline{t}, n, \overline{n}\}$,
    \item $\prescript{n}{m}\Popt(w |_{n-1})^{-1}(t) = n$,
    \item $\prescript{n}{m}\Popt(w |_{n-1})^{-1}(\overline{t}) = \overline{n}$,
    \item $\prescript{n}{m}\Popt(w |_{n-1})^{-1}(n) = \Popt^B(w |_{n-1})^{-1}(\overline{t})$, and
    \item $\prescript{n}{m}\Popt(w |_{n-1})^{-1}(\overline{n}) = \Popt^B(w |_{n-1})^{-1}(t)$.
\end{itemize}
\end{definition}

Recall the definition of $w_1$ and $w_2$ from before. Observe that if $\prescript{n}{}\Popt(w |_{n-1}) = w_i$, then $ \prescript{n}{m}\Popt(w |_{n-1}) = w_{3-i}$. In particular, if $\pi_T(w)$ is zeroed and $\pi_T(\prescript{n}{}\Popt(w |_{n-1})) \neq \pi_T(w)$, then $\pi_T(\prescript{n}{m}\Popt(w |_{n-1})) = \pi_T(w)$. 

This shows that there are only the following two possibilities if $w \in D_n$ and $\pi_T(w)$ is zeroed is such that $\Popt^B(w|_{n-1}) \neq \Popt(w)|_{n-1}$. 

\begin{enumerate}
    \item $\pi_T(\Popt(w))$ is not zeroed. Lemma~\ref{lem:cover-Dn} then implies that the zero block of $\pi_T(w)$ is split into two cycles in $\pi_T(\Popt(w))$. Call the corresponding $w$ \emph{double-cysted}.
    \item $\pi_T(\Popt(w))$ is zeroed and $\Popt^2(B) = \prescript{n}{m}(\Popt^B(w|_{n-1}))$. Call such an element $w$ \emph{pierced}.
\end{enumerate}

\begin{figure}[!htp]
    \centering
    \begin{tikzpicture}
\begin{scope}
\equib[1.8]{4}
\draw[-{Stealth[length=2mm, width=1mm]}, line width = 0.7pt] (N1) -- (N4);
\draw[-{Stealth[length=2mm, width=1mm]}, line width = 0.7pt] (N4) -- (P1);
\draw[-{Stealth[length=2mm, width=1mm]}, line width = 0.7pt] (P1) -- (P4);
\draw[-{Stealth[length=2mm, width=1mm]}, line width = 0.7pt] (P4) -- (N1);

\draw[-{Stealth[length=2mm, width=1mm]}, line width = 0.7pt] (N2) -- (N3);
\draw[-{Stealth[length=2mm, width=1mm]}, line width = 0.7pt] (N3) -- (P2);
\draw[-{Stealth[length=2mm, width=1mm]}, line width = 0.7pt] (P2) -- (P3);
\draw[-{Stealth[length=2mm, width=1mm]}, line width = 0.7pt] (P3) -- (N2);

\end{scope}
\begin{scope}[shift={(6,0)}]
\equid[1.8]{4}
\draw[-{Stealth[length=2mm, width=1mm]}, line width = 0.7pt] (N1) -- (N4);
\draw[-{Stealth[length=2mm, width=1mm]}, line width = 0.7pt] (N4) -- (P1);
\draw[-{Stealth[length=2mm, width=1mm]}, line width = 0.7pt] (P1) -- (P4);
\draw[-{Stealth[length=2mm, width=1mm]}, line width = 0.7pt] (P4) -- (N1);

\draw[-{Stealth[length=2mm, width=1mm]}, line width = 0.7pt] (N2) -- (N3);
\draw[-{Stealth[length=2mm, width=1mm]}, line width = 0.7pt] (N3) -- (ori);
\draw[-{Stealth[length=2mm, width=1mm]}, line width = 0.7pt] (ori) -- (N2);
\draw[-{Stealth[length=2mm, width=1mm]}, line width = 0.7pt] (P2) -- (P3);
\draw[-{Stealth[length=2mm, width=1mm]}, line width = 0.7pt] (P3) -- (ori);
\draw[-{Stealth[length=2mm, width=1mm]}, line width = 0.7pt] (ori) -- (P2);

\end{scope}
\end{tikzpicture}
    \caption{On the left, we have $w|_{4} =(1 \enspace 4 \enspace \overline{1} \enspace \overline{4}) (2 \enspace 3 \enspace \overline{2} \enspace \overline{3})$. The $5$-extension of $w|_{4}$ is given by $(1 \enspace 4 \enspace \overline{1} \enspace \overline{4}) (2 \enspace 3 \enspace 5)(\overline{2} \enspace \overline{3} \enspace \overline{5})$.}
    \label{fig:extensions}
\end{figure}
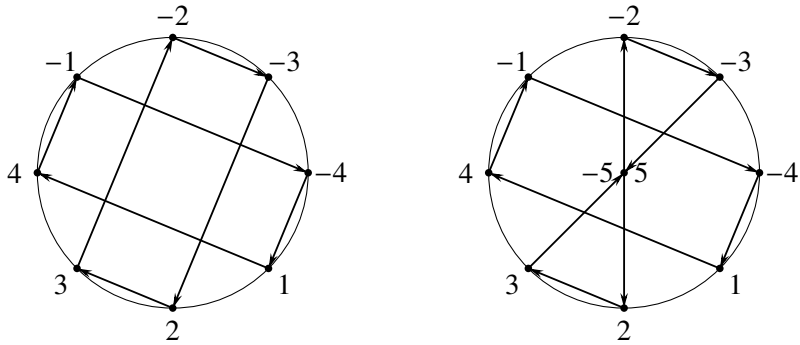

\begin{definition}
$w \in D_n$ is \emph{bad} if either $w$ is double-cysted or pierced.
\end{definition}

\begin{corollary}\label{cor:red-to-B}
If $w \in D_n$ is such that $\pi_T(w)$ is zeroed and such that $\Popt^i(w)$ is not bad for $1 \leq i \leq t$. We have the following alternative characterization of $\Popt^i(w)$ for $1 \leq i \leq t$.
\begin{itemize}
    \item If $i$ is odd, then $\Popt^i(w) = \prescript{n}{}\Popt( (\Popt^B)^{i-1}(w|_{n-1}))$. 
    \item If $i$ is even, then $\Popt^i(w) = \prescript{n}{m}\Popt( (\Popt^B)^{i-1}(w|_{n-1}))$. 
\end{itemize}
\end{corollary}

Next, we examine bad elements in greater detail. First, we will study double-cysted elements. 
    
\begin{claim}\label{claim:double-cyst}
Suppose $w$ is double-cysted, so $ \mathcal{C}_1 \mathcal{C}_2 \subset \pi_T(\Popt(w))$ for some (non-balanced) cycles $\mathcal{C}_1 \ni n$ and $\mathcal{C}_2 = \overline{\mathcal{C}_1} \ni \bar{n}$. By re-indexing $\bar{i}$ as $i$ in each of $\mathcal{C}_1$ and $\mathcal{C}_2$ as necessary, we can consider $\mathcal{C}_1$ and $\mathcal{C}_2$ as elements of $A_n$. Then $\min\{ \aexc^A(\mathcal{C}_1), \aexc^A(\mathcal{C}_2)\} \leq \aexc^B(\Popt(w|_{n-1}))$. 
\end{claim}

\begin{proof}
Let $\mathcal{C}_1$ and $\mathcal{C}_2$ after re-indexing be $\mathcal{D}_1$ and $\mathcal{D}_2$, respectively. Note that by re-indexing it is possible that we create ``new'' antiexceedances of the form $\mathcal{D}_j^{-1}(u) = v$ where $v > u$ but  $\mathcal{C}_j^{-1}(u) = \bar{v}$. without loss of generality $j = 1$. Then note that $\mathcal{D}_2^{-1}(v) = u$ while $\mathcal{C}_2^{-1}(\bar{v}) = u$ so an antiexceedance in $\mathcal{C}_2$ is ``destroyed'' when we pass to $\mathcal{D}_2$ as a result of re-indexing. Consequently, we have that $\aexc^B(\Popt(\Popt^B(w|_{n-1}))) \geq \max \{ (\aexc(\mathcal{C}_1) - 1) + (\aexc(\mathcal{C}_2 - 1), 1 \}$. Here, the $-1$ in $\aexc(\mathcal{C}_1) - 1$ accounts for the additional antiexceedance from $\mathcal{C}_1^{-1}(n)$ that is not present in $w|_{n-1}$. In particular, if $\max\{\aexc(\mathcal{C}_1), \aexc(\mathcal{C}_2) \} -1 \geq 1$ then $\aexc(\Popt(\Popt^B(w|_{n-1}))) \geq \min\{ \aexc(\mathcal{C}_1), \aexc(\mathcal{C}_2)\}$. Otherwise, we have that \[\aexc(\mathcal{C}_1) = \aexc(\mathcal{C}_2) = \aexc(\Popt(\Popt^B(w|_{n-1}))) = 1.\] 
\end{proof}

As a consequence of the above claim and Lemma~\ref{lem:DN-A}, we have the following useful consequence. 

\begin{corollary}\label{cor:dc-aexc}
Suppose $w$ is double-cysted, then $\left|O_{\Popt(w)} \right| \leq \aexc(\Popt^B(w|_{n-1}))$.
\end{corollary}

\begin{proof}
Recall by Lemma~\ref{lem:DtoB} that $\Popt(w) = \prescript{n}{}\Popt(w|_{n-1})$. In particular, any cycle $\mathcal{C}' \in w$ with $\mathcal{C}' \cap \{n , \bar{n} \} = \emptyset$ satisfies $\mathcal{C}' \in \Popt^B(w|_{n-1})$ as well. In particular, write $\pi_T(w|_{n-1}) = \mathcal{C} \mathcal{C}_1 \cdots \mathcal{C}_j$. Note that $\sum_{i=1}^{j} \aexc{\mathcal{C}_i} \leq \aexc(\Popt^B(w|_{n-1}))$, and so Lemma~\ref{lem:DN-B} implies that $\max_j \left\{  \left| O_{\Popt^B}(\mathcal{C}_j) \right| \right \} \leq \aexc(\Popt^B(w|_{n-1}))$.

By Lemma~\ref{lem:cover-Dn}, note that for $j \geq 1$ we have $\pi_T(\Popt^j(w))$ further refines each of $\mathcal{D}_1, \mathcal{D}_2, \mathcal{C}_1, \ldots \mathcal{C}_j$. That is, the dynamic of $\Popt$ can be thought of acting independently on each of the cycles of $w$. Making the identification of $\bar{i}$ with $i$ as described in Claim~\ref{claim:double-cyst} when necessary, it follows from the previous paragraph as well as the result of Claim~\ref{claim:double-cyst} that \[\left| O_{\Popt}(w) \right| = \max_j\left\{ \left| O_{\Popt^A}(\mathcal{D}_1) \right|, \left| O_{\Popt^A}(\mathcal{D}_2) \right|, \left| O_{\Popt^B}(\mathcal{C}_j) \right| \right\} \leq \aexc(w|_{n-1}).\]
\end{proof}

Next, we make an observation about the uniqueness of pierced permutations in the $\Popt$ orbit.

\begin{lemma}\label{lem:pierced}
Suppose $w \in D_n$ is such that $w$ is pierced. Then there does not exist $i \geq  1$ such that $\Popt^i(w)$ is pierced as well.
\end{lemma}

In the proof of the lemma we will repeatedly make use of the notion of \emph{cycle disconnected}. 

\begin{definition}
Given $w \in D_n$, we say that cycles $\mathcal{C}_1, \mathcal{C}_2$ are \emph{cycle disconnected} (in $w$) if there do not exist cycles $\mathcal{D}_1, \ldots, \mathcal{D}_k$ such that the convex hulls of $\mathcal{D}_1$ and $\mathcal{C}_1$ have nonempty intersection in their interiors, the convex hulls of $\mathcal{D}_k$ and $\mathcal{C}_2$ have nonempty intersection in their interiors and for $1 \leq i \leq k-1$ the convex hulls of $\mathcal{D}_i$ and $\mathcal{D}_{i+1}$ have nonempty intersection in their interiors. 

Conversely, if there exist such cycles $\mathcal{D}_1, \ldots, \mathcal{D}_k$ we say that $\mathcal{C}_1$ and $\mathcal{C}_2$ are \emph{cycle connected}.
\end{definition}

Observe that for $w \in D_n$ and cycles $\mathcal{C}_1$ and $\mathcal{C}_2$ that are cycle disconnected in $w$, we necessarily have that the elements of $\mathcal{C}_1$ and the elements of $\mathcal{C}_2$ lie in distinct cycles in $\pi_T(w)$. Conversely, if $\mathcal{C}_1$ and $\mathcal{C}_2$ are cycle connected then their elements lie in the same cycle in $\pi_T(w)$. 

\begin{proof}
Suppose otherwise for the sake of contradiction, and suppose $i$ is the smallest such counterexample. Firstly, note that if there exists $j \geq 1$ such that $\Popt^j(w)$ is double-cysted, then $\pi_T(\Popt^k(w))$ cannot be zeroed for $k \geq j$ by Lemma~\ref{lem:cover-Dn}. In particular, $\Popt^k(w)$ cannot be pierced either. Consequently, we may assume for $1 \leq k \leq i$ that $\Popt^k(w)$ is not bad.

Let $w(n) = a$. Suppose $\Popt^B(w|_{n-1})^{-1}(a) = b$. Define $u \in B_{n-1}$ as follows:
\[ u(x) = \begin{cases} \Popt^B(w|_{n-1})(x) &\text{if } x \neq b, \bar{b}, \\ \bar{a} &\text{if } x = b, \\ a &\text{if } x = \bar{b}. \end{cases} \]
Then the fact that $w$ is pierced is equivalent to $\Popt(w)|_{n-1} = u$. In particular, we have that $\Popt^2(w) = \prescript{n}{}\Popt(\Popt(w)|_{n-1}) = \prescript{n}{}\Popt(u)$. By the previous paragraph and Corollary~\ref{cor:red-to-B}, the behavior of $\Popt^k(w)$ for $1 \leq k \leq i$ depends only on $(\Popt^B)^k(u)$.

Furthermore, since $w$ is pierced, it follows that there exists a cycle $\mathcal{C}$ that is cycle disconnected from the cycle $\mathcal{B}_1$ in $\Popt^B(w|_{n-1})$ that contains $a$, but cycle connected from the cycle $\mathcal{B}_2$ in $u$ that contains $a$. 

Let $\mathcal{C}_1$ be the cycle in $\Popt^{i+1}(w)$ which contains $a$. Let $\mathcal{C}_2$ be the cycle in $(\Popt^B)^{i+1}(u)$ which contains $a$, and let $b' = (\Popt^B)^{i+1}(u)^{-1}(a) $. Note that the cycles in $\Popt^{i+1}(w)$ and $(\Popt^B)^{i+1}(u)$ that do not contain any of $\{a, \bar{a} \}$ overlap. In particular, in order for $\Popt^i(w)$ to be pierced, there must be cycles $\mathcal{D} \in \Popt^{i+1}(w), (\Popt^B)^{i+1}(u)$ such $\mathcal{D}$ and $\mathcal{C}_1$ are cycle connected in $\Popt^{i+1}(w)$, while $\mathcal{D}$ and $\mathcal{C}_2$ are cycle disconnected in $(\Popt^B)^{i+1}(u)$. 

We split into two cases based on the clockwise orientation of $a,\bar{b}, \bar{a}, b$. First, suppose we are in the situation where in clockwise order on the circle we have $\bar{a}, b, a, \bar{b}$. Note that we may choose $\mathcal{C}$ such that it only contains elements that either lie clockwise between $\bar{a}$ and $b$, or lie clockwise between $a$ and $\bar{b}$. We claim that $\mathcal{C}$ cannot contain an element $c$ such that $c$ is clockwise between $\bar{a}$ and $b$ and $d = \mathcal{C}(c)$ lies clockwise between $a$ and $\bar{b}$, which would then contradict the previous observation on $\mathcal{C}$. This is because if we let $\widetilde{\mathcal{C}}$ be the cycle in $\Popt^{i+1}(w)$ that contains $d$, then $d$ would also be situated clockwise between $a$ and $\bar{b'}$, while $\widetilde{\mathcal{C}}^{-1}(d)$ is situated clockwise between $\bar{a}$ and $b'$. In particular, the convex hull of $\widetilde{\mathcal{C}}$ has nonempty interior intersection with the convex hull of $\mathcal{C}_2$, and $\widetilde{\mathcal{C}}$ would also have nonempty interior intersection with the convex hull of $\mathcal{D}$, so that $\mathcal{D}$ and $\mathcal{C}_2$ are cycle connected in $(\Popt^B)^{i+1}(u)$, a contradiction.

Next, suppose in clockwise order on the circle we have $a,b,\bar{a}, \bar{b}$. Similar to before, note that we may pick $\mathcal{C}$ to contain an element $c$ that is situated clockwise between $b$ and $\bar{a}$ such that $\mathcal{C}(c) = d$ is situated clockwise between $\bar{b}$ and $a$. We may also choose $\mathcal{D}$ such that it contains an element $c'$ that is situated clockwise between $\bar{a}$ and $\bar{b'}$ such that $d' = \mathcal{D}(c')$ is situated clockwise between $a$ and $b$. Let $\mathcal{C}'$ be the cycle in $w$ that contains $d'$. Note that if $\mathcal{C}'$ only consists of elements that are either clockwise between $\bar{a}$ and $\bar{b}$ or clockwise between $a$ and $b$ we would obtain our desired contradiction because $\mathcal{C}$, $\mathcal{C}'$ and the $ \mathcal{B}_1$ shows that $\mathcal{C}$ and $\mathcal{B}_1$ are cycle connected. 

Otherwise, suppose $d'$ lies clockwise between $a$ and $b$ while $(\mathcal{C}')^{-1}(d') = f'$ lies clockwise between $b$ and $\bar{a}$. In this case, observe that the convex hulls of $\mathcal{C}'$ and $\mathcal{B}_1$ have non-empty intersections in their interior, and the same conclusion also holds for $\overline{\mathcal{C}'}$ and $\mathcal{B}_1$. As such, in order for $\mathcal{C}$ to be cycle disconnected from $\mathcal{B}_1$, it follows that $\mathcal{C}$ can only contain elements that are situated clockwise between $f'$ and $\bar{a}$ or elements that are clockwise between $\bar{f'}$ and $a$ on the circle. Suppose $\mathcal{C}$ contains $g$, which lies clockwise between $f'$ and $\bar{a}$ such that $\mathcal{C}^{-1}(g)$ lies clockwise between $\bar{f'}$ and $a$. Let $\widetilde{\mathcal{C}}$ be the cycle in $\Popt^{i+1}(w)$ that contains $g$. Note that the convex hulls of $\widetilde{\mathcal{C}}$ and $\mathcal{C}_2$ have non-empty intersection in their interiors, and the convex hulls of $\widetilde{\mathcal{C}}$ and $\mathcal{D}$ have non-empty intersection in their interiors. It follows that $\mathcal{D}$ and $\mathcal{C}_2$ are cycle connected, which is a contradiction.

Lastly, consider the case where both $d'$ and $(\mathcal{C}')^{-1}(d') = f'$ lie clockwise between $b$ and $\bar{a}$. Note that in order for $\mathcal{C}'$ to be cycle connected to $\mathcal{B}_2$, there has to be a cycle $\mathcal{B} \in u$ which contains an element $\tilde{a}$ that lies clockwise between $d'$ and $f'$ such that $\mathcal{B}^{-1}(\widetilde{a}) = \tilde{b}$ lies clockwise between $\bar{a}$ and $b$. Let $\mathcal{B}'$ be the cycle in $\Popt^i(w)$ that contains $\tilde{a}$. Then $\mathcal{D}$, $\mathcal{B}'$ and $\mathcal{C}_2$ shows that $\mathcal{D}$ and $\mathcal{C}_2$ are cycle connected in $(\Popt^B)^{i+1}(u)$, which is a contradiction.
\end{proof}

The above observations allow us to reduce to the case of $w$ being pre-$D$-splittable. 

\begin{lemma}\label{lem:d-split}
If $\left| O_{\Popt}(w) \right| = 2n-2$, then $w$ is pre-$D$-splittable.
\end{lemma}

\begin{proof} 

Suppose $w$ is such that $\left| O_{\Popt}(w) \right| = 2n-2$. By Lemma~\ref{lem:not-zeroed}, it follows that $w$ is zeroed. Suppose for the sake of contradiction that $w$ is not pre-$D$-splittable. Then it follows that $\aexc(w|_{n-1}) \leq 2n-5$. If there does not exist $i$ such that $\Popt^i(w)$ is bad, then by Corollary~\ref{cor:red-to-B}, it follows that $\left| O_{\Popt}(w) \right| = \left| O_{\Popt^B}(w|_{n-1}) \right| \leq 2n-5 + 1 = 2n-4$, which is a contradiction. The remaining situations to handle are when there is some $i$ such that $\Popt^i(w)$ is bad.

First, suppose there exists $j$ such that $\Popt^j(w)$ is pierced. Lemma~\ref{lem:pierced} implies that $j$ is unique. By Lemma~\ref{lem:DN-B} as well as the arguments before this lemma, we have that \[\left|O_{\Popt^B}(\Popt^{j+1}(w)|_{n-1})\right| \leq \aexc(\Popt^{j+1}(w)|_{n-1}) \leq 2n - 5 - j.\] 

Finally, suppose there exists $i$ such that $\Popt^i(w)$ is double-cysted. By Corollary~\ref{cor:dc-aexc}, since $\Popt^i(w)|_{n-1} = (\Popt^B)^i(w|_{n-1})$, we have that the number of additional iterations of $\Popt$ needed for $\Popt^i(w)$ to reach $e$ does not exceed the number of additional iterations of $\Popt^B$ needed for $(\Popt^B)^i(w|_{n-1})$.

Combining the above, it follows that $\left| O_{\Popt}(w) \right| \leq \left| O_{\Popt^B}(w|_{n-1}) \right| + 1 \leq 2n-5 +2 = 2n-3$, a contradiction. \end{proof}

Before we can finish up the proof of Theorem~\ref{thm:type-D}, we will need a simple observation about permutations in $B_n$.

\begin{claim}\label{claim:tot-pierced}
Suppose $w \in B_n$ is such that $\cyc(\pi_T(\Popt^i(w))) \leq 1$ for all $ i \geq 0$. Suppose also that $(x)(\overline{x}) \in w$. Define $v \in B_n$ by 
\[ v(y) = \begin{cases} w(y) &\text{if } y \not \in \{x, \overline{x} \}, \\ \overline{x} &\text{if } y = x, \\ x &\text{if } y = \overline{x}. \end{cases}\]
Then $\left|O_{\Popt}(v)\right| = \left| O_{\Popt}(w) \right| + 1$. 
\end{claim}

\begin{proof}
Define a total order $\prec$ on $\pm[n-1]$ by $\overline{1} \prec \cdots \prec \overline{n-1} \prec 1 \prec \cdots \prec n-1$. Define $\prec_j$ to be the order obtained by doing a cyclic shift on $\prec$ such that the maximum element of $\prec_j$ is $j$. 

Note that $\aexc{v} = \aexc(w) + 1$. By Lemma~\ref{lem:DN-B}, it suffices to show that $\cyc(\pi_T(\Popt^i(v))) \leq 1$ for all $i \geq 0$ as well. To that end, let $\mathcal{A} = \{y: y \prec_x \bar{y}\} $ note that the convex hull of $\{y, \Popt^i(v))(y): y \in \mathcal{A}\} \cup \{ x, \Popt^i(v))(x) \} \setminus \{ \bar{x} \}$ coincides with that of $\{y, \Popt^i(w)(y): y \in \mathcal{A}\}$. As such, $\cyc(\pi_T(\Popt^i(v))) \leq 1$ follows from $\cyc(\pi_T(\Popt^i(w))) \leq 1$. 
\end{proof}

\begin{proof}[Proof of Theorem~\ref{thm:type-D}]

First, let us establish that in order for $w \in D_n$ to satisfy $|O_{\Popt}(w)| = 2n-2$, $w$ must be $D$-splittable. Lemma~\ref{lem:d-split} already shows that $w$ must be pre-$D$-splittable. Our work in Section~\ref{sec:B} shows that $\cyc( \pi_T((\Popt^B)^k (w|_{n-1}))) = 1$ for all $i \geq 0$. Let $x = w(n)$. Note that since $\cyc( \pi_T((\Popt^B)^k (w|_{n-1}))) = 1$, the only possibility for $\Popt^j(w)$ to be pierced is if $(x \enspace \bar{x}) \in \Popt^{j+1}(w)|_{n-1}$ but $(x)(x \enspace \bar{x}) (\Popt^B)^{j+1}(w|_{n-1})$. Let this property of $w$ be ($\dagger$). 

Suppose $w$ is pre-$D$-splittable but not $D$-splittable. Then the smallest index such that \[ (x \enspace \bar{x}) \in (\Popt^B)^{j+1}(w|_{n-1})\]
is $j = 2n-4$. But $(\Popt^B)^{j+1}(w|_{n-1}) = e$. This shows that there does not exist $i$ such that $\Popt^i(w)$ is pierced. Corollary~\ref{cor:red-to-B} and Claim~\ref{claim:double-cyst} combined show that $|O_{\Popt}(w)| = 2n-3$, a contradiction.

It follows that to complete the proof of theorem it remains to show that $D$-splittable permutations indeed have $\Popt$ forward orbit of length $2n-2$ and to also enumerate the number of such pre-$D$-splittable permutations.

Define a total order $\prec$ on $\pm[n-1]$ by $\overline{1} \prec \cdots \prec \overline{n-1} \prec 1 \prec \cdots \prec n-1$. As before, define $\prec_j$ to be the order obtained by doing a cyclic shift on $\prec$ such that the maximum element of $\prec_j$ is $j$. 

First, we observe some properties of pre-$D$-splittable permutations. Suppose $w \in D_n$ is pre-$D$-splittable, and let $w(n) = i$. For simplicity of notation, suppose $i > 0$; the analysis for when $i < 0$ is analogous. First, it is easy to see that $w|_{n-1}^{-1}(i) \prec_{i} \overline{i+1}$. Note that an alternate description of permutations $v \in B_{n-1}$ with $2n-4$ antiexceedances is that it is made up of two cycles whose elements are arranged in counterclockwise order on the circle of elements $\pm[n-1]$. In particular, it follows that $\{ x: x \prec_i w|_{n-1}^{-1}(i) \}$ lie in the same cycle. When $w|_{n-1}^{-1}(i) \prec_{i} \overline{i+1}$, it can be checked that the cycle in $w|_{n-1}$ containing $w|_{n-1}^{-1}(i)$ cannot have the property that the elements are counterclockwise on the circle, which would contradict $w|_{n-1}$ containing $2n-4$ antiexceedances. We also claim that $w|_{n-1}^{-1}(i) \neq \overline{i-1}$. This is because the unique permutation $u$ that has $2n-4$ antiexceedances satisfying $u^{-1}(i) = i-1$ is \[(i \enspace i+1 \enspace \cdots \enspace n-1 \enspace \bar{1} \enspace \bar{2} \enspace \cdots \enspace \overline{i-1}) (\overline{i} \enspace \overline{i+1} \enspace \cdots \enspace \overline{n-1} \enspace 1 \enspace 2 \enspace \cdots i-1)\] and $\cyc(\pi_T(u)) > 1$.

Now, we enumerate the number of $D$-splittable permutations. Note that since $i = w(n) \in \pm[n-1]$, there are a total of $2(n-1)$ possible choices for $w(n)$. By symmetry, we assume that $i > 0$. Observe that by definition of the $(n-1)$-projection, any $D$-splittable permutation has the property of $w|_{n-1}^{-1}(i) = w^{-1}(i)$. Suppose $w^{-1}(i) = j$ where $j \prec_i \bar{i-2}$. The elements $\{ x: i \prec_{j} x \}$ all lie in a cycle. The remaining choices we have to make is to label each of the elements of $\{ y: j-1 \prec_{\overline{i}} y \}$ with either $C_1$ or $C_2$ (indicating which cycle the element lies in). Summing over all possible choices of $j$, there are $2^{n-4} + 2^{n-5} + \cdots + 2 = 2^{n-3} - 2$ way to do the labelling. Note that $ j= \bar{i}$ gives rise to one more such permutation, for a total of $2^{n-3} - 1$ possible permutations. Accounting for the $2(n-1)$ choices for $i$, as well as $c^{-1}$, in total the number of $D$-splittable elements is given by $2(n-1) \cdot (2^{n-3} + 1) + 1 = (n-1)\cdot (2^{n-2} +2) + 1$, as desired. 

It remains to prove that pre-$D$-splittable permutations $w$ actually have $\Popt$ forward orbits of length $2n-2$. It is clear that $\left| O_{\Popt}(c^{-1})\right| = 2n-2$, so it remains to handle the case when $w \neq c^{-1}$. Our task is two-fold: first, we will show that for all $i \geq 1$, $\Popt^i(w)$ is zeroed. This would in particular imply that there does not exist $i$ such that $\Popt^i(w)$ is double-cysted. Then, we show that there exists an index $j$ such that $\Popt^j(w)$ is pierced. Let $x = w(n)$. Equivalently, the property ($\dagger$) implies that it suffices to show that for some $j \leq 2n-5$ that $(x)(x \enspace \bar{x}) (\Popt^B)^{j+1}(w|_{n-1})$. Once we have established both of these, the desired conclusion follows from Claim~\ref{claim:double-cyst}, Lemma~\ref{lem:pierced} and Claim~\ref{claim:tot-pierced} that $\left| O_{\Popt}(c^{-1})\right| = (2n-3) + 1 = 2n-2$. 

Now, let us show that for all $i \geq 1$, $\Popt^i(w)$ is zeroed. Observe that any $D$-splittable permutation $w$ has the property that $\left|\{ y: y \prec_x w^{-1}(x) \} \right| \leq n-1$ for $x \in \pm[n-1]$. In particular, for $1 \leq j \leq n-1$ and $x \in \pm[n-1]$, $\Popt^j(w)(x) \neq x$. Suppose $i$ is the smallest index such that $\Popt^i(w)$ is not zeroed. Then there exists some $z \in [n-1]$ such that for all $x \in \pm[n-1]$ with $x \prec_z \bar{z}$, we have that $\Popt^i(w)(x), \Popt^i(w)^{-1}(x)\prec_z \bar{z}$ as well. Let $z' \prec_z \bar{z}$ be the smallest index such that $\Popt^i(w)(z') \neq z'$. Then $\Popt^i(w)(z') \prec_z \bar{z}$ implies that $i < n-1$. This in turn implies that there exists $y$ such that $ y \prec_{\overline{z-1}} w^{-1}(\overline{z-1}) =: a$. Let $b$ be the element that is clockwise along the circle from $a$. Then because of the structure of $D$-splittable permutations, either $w^{-1}(y) = a$ or $w^{-1}(y) = b$. In either case, $\Popt^{j}(w)^{-1}(y) \prec_{z} \bar{z}$. This shows that the convex hull of $\{ \Popt^j(w)(x): x \prec_{z} \bar{z} \}$ has non-empty intersection with $\{x: \bar{z} \prec_z x\}$, which implies that $\Popt^i(w)$ is zeroed, as desired.

Finally, we show that there exists an index $j \leq 2n-5$ such that $(x)(x \enspace \bar{x}) \in (\Popt^B)^{j+1}(w|_{n-1})$. Note that $ w^{-1}(x) \prec_x (\Popt^B)(w)^{-1}(x) \prec_x (\Popt^B)(w)^{-2}(x) \prec_x \cdots$. Since there are at most $2n-5$ elements $y$ such that $w^{-1}(x) \prec_x y$, it follows that for some $j \leq 2n-5$, we have $x = (\Popt^B)^{j+1}(w|_{n-1})^{-1}(x)$, which is the desired conclusion.
\end{proof}

\section{Conclusion and further work}

For a finite Coxeter group $(W,S)$ with Coxeter number $h$, fixed Coxeter element $c$, and set of all reflections $T$, the length of the longest orbit of $\Popt(\cdot, c)$ is $h$ (see for instance \cite[Theorem 5.1 and Theorem 6.1]{DN21}). In the previous sections, we have given a classification and enumeration of the forward orbits with close to maximal length. While we now have a complete understanding of the elements in type $A$ and type $B$ Coxeter groups with forward orbit of length $h-1$, where $h$ is the Coxeter number of the group under consideration, we only understand elements with maximum forward orbit length for type $D$ Coxeter groups. We related the problem of classifying elements with forward orbits of length $2n-3$ in $D_n$ to the question of classifying elements with forward orbits of length $2n-2$ and $2n-3$ in $B_{n-1}$. It seems reasonable to expect that in order to understand forward orbits of length $2n-4$ in $D_n$, it would be beneficial to understand forward orbits of length at most $2n-4$ in $B_{n-1}$.

It is natural, therefore, to ask a similar enumeration and classification question for orbits of other lengths. 

\begin{question}
Let $W$ be a Coxeter group of type $A$ or $B$. Let $h$ be the Coxeter number of $W$ and let $\alpha$ be an integer in the interval $[2, h-2]$. What is the number of forward orbits of $\Popt$ with length $\alpha$? 
\end{question}

\begin{question}
Let $W$ be a Coxeter group of type $D$. Let $\alpha$ be an integer in the interval $[2, 2n-4]$. What is the number of forward orbits of $\Popt$ with length $\alpha$? 
\end{question}

Another potential direction is inspired by \cite{CG19, D22}. In order to state this question, we need to introduce a piece of notation. We say that $w \in W$ is \emph{$t$-pop-stack-sortable} if $\Pops^t(w) = e$. In \cite{CG19}, Claesson and Gomundsson proved that for every fixed $t \geq 0$, the generating function that counts $t$-pop-stack-sortable permutations in symmetric groups is rational. Defant \cite{D22} generalized this result to hyperoctahedral groups $B_n$. It is natural to ask a similar question in the context of $t$-pop-tsack-torsable permutations, defined analogously. We say that $w \in W$ is \emph{$t$-pop-tsack-torsable} if $\Popt^t(w) = e$. 

Note, however, that the situation is slightly more subtle in the context of $\Popt$ as compared to that of $\Pops$. In particular, it is false in general that the generating function which counts $t$-pop-tsack-torsable permutations is rational. For the sake of simplicity, to illustrate this point, we restrict to the context of $W = \mathfrak{S}_n$. By construction, the elements of $\mathfrak{S}_n$ that are $1$-pop-tsack-torsable are exactly the non-crossing partitions of $\mathfrak{S}_n$, and these are counted by the Catalan numbers. It follows, therefore, that the generating function that counts 1-pop-tsack-torsable permutations in $\mathfrak{S}_n$ is not rational. 

Theorem~\ref{thm:type-A}, for instance, shows that the generating function that counts $(n-1)$-pop-tsack-torsable permutations in $\mathfrak{S}_n$ is rational. This observation motivates the following question. 

\begin{question}
Let $W$ be a Coxeter group of type $A,B$, or $D$. What is the range of values of $t$ such that the generating function counting $t$-pop-tsack-torsable is rational? More precisely, what is the range of values of $t$ such that 
\[ \sum_{n \geq 1} \left| \Popt^{-(n-t)}(e) \right| z^n\]
is rational?
\end{question}

\section{Acknowledgements}
This work was done at the University of Minnesota Duluth with support from Jane Street Capital, the NSA (grant number H98230-22-1-0015), and Massachusetts Institute of Technology. We would like to thank Joe Gallian for organizing the REU and for his constant support and helpful suggestions. We are also immensely grateful to Colin Defant for proposing this problem and for his invaluable help in editing this paper. We would also like to thank Yelena Mandelshtam for her helpful comments on a draft version of this paper. Last but not least, we would like to express our appreciation towards Amanda Burcroff, Colin Defant and Noah Kravitz for advising this program.

\bibliographystyle{amsplain}
\bibliography{ref.bib}

\end{document}